\documentclass[10pt]{amsart}
\usepackage{amsmath,amsfonts, mathtools,amssymb,latexsym,amssymb,amsthm,adjustbox,mathrsfs,amsbsy, bm,tikz-cd,indentfirst, makecell,graphicx, verbatim, hyperref, enumerate,xy}

\usepackage[top=1in,bottom=1.3in,left=1.1in,right=1.1in]{geometry}

\theoremstyle{plain}
\newtheorem{theorem}{Theorem}[section]
\newtheorem{lemma}[theorem]{Lemma}

\newtheorem{corollary}[theorem]{Corollary}
\newtheorem{proposition}[theorem]{Proposition}
\newtheorem{remark}[theorem]{Remark}
\newtheorem{definition}[theorem]{Definition}

\newtheorem{example}[theorem]{Example}
\newtheorem{setup}[theorem]{Setup}
\newtheorem{point}[theorem]{}

\newcommand{\m}{\mathfrak{m}}

\newcommand{\Z}{\mathbb{Z}}
\newcommand{\N}{\mathbb{N}}

\newcommand{\n}{\mathfrak{n}}

\newcommand{\codim}{\operatorname{codim}}
\newcommand{\grade}{\operatorname{grade}}
\newcommand{\depth}{\operatorname{depth}}

\newcommand{\ann}{\operatorname{ann}}
\newcommand{\cx}{\operatorname{cx}}

\newcommand{\CMS}{\operatorname{\underline{CM}}}

\newcommand{\pdim}{\operatorname{pdim}}

\newcommand{\Hom}{\operatorname{Hom}}

\newcommand{\Ext}{\operatorname{Ext}}
\newcommand{\Tor}{\operatorname{Tor}}

\usepackage{graphicx}
\usepackage{subfigure}
\usepackage{epsfig}
\usepackage{epstopdf}
\usepackage{float}
\theoremstyle{plain}

\usepackage{xcolor}
\usepackage{titlesec}

\titleformat{name=\section}{}{\thetitle.}{0.5em}{\centering\scshape\large}

\title[DERIVED FUNCTORS AND HILBERT POLYYNOMIALS OVER GORENSTEIN RINGS ]{\small DERIVED FUNCTORS AND HILBERT POLYYNOMIALS OVER GORENSTEIN RINGS}

\author{\textsc{Satyabrata Paul, Tony J. Puthenpurakl}}
\address{Department of Mathematics, IIT Bombay, Powai, Mumbai 400076, India}
\email{satyabratapaul2357@gmail.com, tputhen@gmail.com}

\date{\today}

\begin{document}

\begin{abstract}
    Let $(A,\m,k)$ be a Gorenstein ring of dimension $d\ge 1$, $N$ a perfect module of dimension $t\ge 1$ and $I$ an ideal of definition of $N$. For a non-free  maximal Cohen-Macaulay (=MCM) $A$-module $M$ and an integer $i\ge 1$, it is well known that the functions $n \mapsto \lambda(\Tor_i^A(M,N/I^{n+1}N))$ and $n \mapsto \lambda(\Ext^i_A(M,N/I^{n+1}N))$ are of polynomial types of degrees $r_i^{I,N}(M)$ and $s_{I,N}^i(M)$, respectively. We prove that $r_i^{I,N}(M)\le t-1$ and $s^i_{I,N}(M)\le t-1$ and when $I$ is the maximal ideal $\m$, both the inequalities become equalities. We also show that $r_i^{I,N}(M)\le r_1^{I,N}(\Omega^dk)$, $s^i_{I,N}(M)\le s^1_{I,N}(\Omega^dk)$ and $r_i^{I,N}(\Omega^dk)=r_1^{I,N}(\Omega^dk)=s^1_{I,N}(\Omega^dk)=s^i_{I,N}(\Omega^dk)$.
\end{abstract}
\maketitle
\section{Introduction} In this paper, all rings considered are commutative, Noetherian, local with unity and all modules considered are finitely generated, unless staed otherwise. We use terminology from \cite{Bruns and Herzog}. Unless otherwise specified, $A$ denotes a local ring with unique maximal ideal $\m$ and residue field $k$. The aim of this paper is to study the degree of Hilbert polynomials of torsion and extension functors. If $X$ is an $A$-module of finite length, we denote its length by $\lambda(X)$. Let $M,N$ be finitely generated $A$-modules $I$ an ideal of $A$. It is well known that, if the lengths $\lambda(M/I^nM)$ are finite for all sufficiently large $n$, then they are given by a polynomial function in $n$, called the \textit{Hilbert–Samuel polynomial}, whose degree equals $\dim M$. Let $i\ge 1$ be an integer. In his paper, Kodiylam proved a generalized version of this result, by showing that $\lambda(\Tor_i^A(M,N/I^nN))$ is eventually polynomial, provided $\lambda(M\otimes N)$ is finite; see \cite[Theorem 2]{Kodi}. Later, Theodorescu proved that, if the lengths  $\lambda(\Tor_i^A(M, N/I^{n+1}N))$ and $\lambda(\Ext^i_A(M, N/I^{n+1}N))$ are finite for all sufficiently large $n$, then these lengths are eventually given by polynomial functions in $n$. Moreover, he derived bounds for the degrees of the corresponding Hilbert–Samuel polynomials and provided a condition under which these bounds are attained; see \cite[Corollary 4]{E.Theodorescu}. However, verifying whether this condition holds is often a difficult task. Also the degrees of the corresponding Hilbert-Samuel polynomials are not as easy to determine; see \cite{IP}, \cite{KT}. There are some results which show under certain contains the maximal degree is attained; see \cite{CKST}, \cite{IP}, \cite{KP}, \cite{KT}, \cite{Hilbertcoefficient} and \cite{PZ}. On the other hand, this function can also be identically zero; see \cite[Remark 20] {Hilbertcoefficient}, \cite[Example 2.6]{PZ}. In this paper, we improve the existing bounds, characterize the cases in which they are achieved, and analyze the situations where they fail to be attained.

\medskip
Recall, a non-zero $A$-module $N$ is said to be \textit{perfect} if $\pdim_AN=\min\{j|\Ext^j_A(M,A)\ne 0\}$. Note that, if the ring $A$ is Cohen-Macaulay, then $N$ is perfect if and only if $N$ is a Cohen-Macaulay  module and $\pdim_AN<\infty$ (see \cite[Theorem 2.1.5]{Bruns and Herzog}).

\medskip

Throughout this paper we work with the following setup, unless stated otherwise.
\begin{setup}   \normalfont \label{setup0} Let $(A,\m)$ be a non-regular Gorenstein local of dimension $d\ge 1$, with the residue field $A/\m =k$.  Let $N$ be a perfect $A$-module of dimension $t\ge 1$ and $I$ an ideal of definition of $N$. For a maximal Cohen-Macaulay (=MCM) $A$-module $M$ and an integer $i\ge 1$, we set \begin{align*}
   r_{i}^{I,N}(M)=\deg (n\mapsto \lambda(\Tor_{i}^{A}(M,N/I^nN))) ,\\
   s^{i}_{I,N}(M)=\deg (n\mapsto \lambda(\Ext^{i}_{A}(M,N/I^nN))).
\end{align*}
\end{setup}
\medskip
Recall $A$ is said to be a \textit{hypersurface} ring if its completion $\widehat{A}=Q/(f)$ where $(Q,\n)$ be a regular local ring and $f\in \n^2$ is non-zero. Motivated by the above discussion, the second author investigated the degrees of the corresponding Hilbert–Samuel polynomials over hypersurface rings. More precisely, he proved the following theorem.
\noindent
\begin{theorem} \cite[Theorem 9.12]{HypersurfaceII} Let $(A,\m)$ be a hypersurface local ring, $N$ a perfect $A$-module and $I$ an ideal of definition $N$. Then  there exist integers $r,s$ depending only on $I$ and $N$ such that for any non-free MCM $A$-module $M$ and integer $i\ge 1$,  one has
\begin{enumerate} [\rm (i)]
    \item $r=r_i^{I,N}(M)$,
    \item $s=s_{I,N}^i(M)$,
    \item $r=s$.
\end{enumerate}

\end{theorem}
\noindent
He has also explored the case when $A$ is complete intersection (see \cite[Theorem 1.1]{CI}).

\medskip

Our motivation is to investigate these results over Gorenstein rings. Our first objective is to improve the known estimates for the numerical invariants $r_i^{I,N}(M)$ and $s_{I,N}^i(M)$. To this end, we show

\medskip

\noindent \textbf{Theorem I.} (Theorem \ref{abound}) [With hypothesis as in \ref{setup0}]\phantomsection\label{Theorem I} \textit{ Let $M$ be a  MCM $A$-module $M$. Then for an integer  $i\ge 1$, we have
\begin{enumerate} [\rm (i)]
    \item $  r_i^{I,N}(\Omega^dk)=r_1^{I,N}(\Omega^dk)$,
    \item $ s^i_{I,N}(\Omega^dk)=s^1_{I,N}(\Omega^dk) $,
    \item $r_1^{I,N}(\Omega^dk)=s^1_{I,N}(\Omega^dk)$,
    \item $r_i^{I,N}(M)\le r_1^{I,N}(\Omega^dk)\le t-1$,
    \item $s^i_{I,N}(M)\le s^1_{I,N}(\Omega^dk)\le t-1 $.
\end{enumerate}
}
\noindent

\medskip
As a consequence of the preceding theorem, the MCM $A$-module $\Omega^dk$ dominates all other MCM $A$-modules in terms of the degrees of the associated Hilbert-Samuel polynomials. It is therefore natural to ask what happens when one considers the higher syzygies and cosyzygies of
$\Omega^dk$. We use  spectral sequences and techniques from MCM approximation (see \cite[Theorem A]{Auslander Buchweitz}) to prove the following theorem.

\medskip

\noindent \textbf{Theorem II.} (Theorem \ref{Omegadk=Omegan+dk}) [With hypothesis as in \ref{setup0}] \phantomsection\label{Theorem II} \textit{Let $n$ be an integer, then we have
\begin{enumerate} [\rm (i)]
\item $r_1^{I,N}(\Omega^dk)=r_1^{I,N}(\Omega^n\Omega^dk)$,
\item $s^1_{I,N}(\Omega^n\Omega^dk)=s^1_{I,N}(\Omega^dk).$
\end{enumerate}
}
\medskip

\noindent We next show that, in the case $I=\m$, the inequalities of \hyperref[Theorem I]{Theorem I} become equalities.
\medskip

\noindent \textbf{Theorem III.} (Theorem \ref{tor,I=m}, \ref{ext,I=m}) [With hypothesis as in \ref{setup0}] \phantomsection\label{Theorem III} \textit{Let $M$ be a non-free MCM $A$-module. Then for any integer $i\ge 1$, we have $ r_{i}^{\m,N}(M)=t-1=s^{i}_{\m,N}(M)$.}

\medskip

\noindent We exhibit an example (see Remark \ref{counterexample}) showing that the conclusion of above theorem fails for arbitrary $\m$-primary ideals.
\medskip

\medskip

Next, we study the stable category MCM $A$-modules $\CMS(A)$ and various thick subcategories of $\CMS(A)$.  We use \cite{Neeman} for notation on triangulated categories. For a MCM $A$-module $M$, we set $r_\infty^{I,N}(M)=\sup \{r_1^{I,N}(\Omega^sM)|s\in \Z\}$. We show that, when $A$ is a complete intersection, the infimum of these values taken over all non-free MCM $A$-modules is attained by a 2-periodic MCM $A$-module. More precisely, we prove

\medskip
\noindent
\textbf{Theorem IV.} (Theorem \ref{cx1CI}) \phantomsection\label{Theorem IV} \textit{Let $(A,\m)$ be a complete intersection local ring of dimension $\ge1$ and $N$ a perfect $A$-module of dimension $\ge1$ and $I$ an ideal of definition of $N$. Then, with the above notation, there exists a 2-periodic MCM $A$-module $E$(that is, $\Omega^2E\cong E$), such that $$r_1^{I,N}(E)=\inf \{r_\infty^{I,N}(M)|0\ne M \in \CMS(A)\}.$$}

\medskip
We continue our discussion by considering the case where the bound in \hyperref[Theorem III]{Theorem III} is not attained. To this end, we take the perfect module $N$ to be $A$ itself. Let $G_I(A)=\bigoplus _{n\ge 0}I^n/I^{n+1}$ be the \textit{associated graded ring} of $A$ with respect to $I   $ and $G_I(M)=\bigoplus _{n\ge 0}I^nM/I^{n+1}M$ be the \textit{associated graded module} of $M$ with respect to $I$, considered as a graded $G_I(A)$-module. Then for each integer $i\ge 0$, we have $H^i_{G_I(A)_+}(G_I(M))_n=0$ for $n\gg 0$ (see\cite[Theorem 16.1.5]{Brodman-Sharp}), where $G_I(A)_+=\bigoplus_{n\ge1}I^n/I^{n+1}$. Set $a_i(G_I(M))=\max \{n|H^i_{G_I(A)_+}(G_I(M))_n\ne0\}$ for an integer $i\ge 0$. For any arbitrary Noetherian local ring $A$ of dimension $d$, Trung proved that $a_d(G_I(A))+d\le \operatorname{red}(I)$ (see \cite[Proposition 3.2]{Trung} and \cite[Theorem 18.3.12]{Brodman-Sharp}) and it can be easily generalized to arbitrary modules of dimension $d$; that is, for any finitely generated module $M$ of dimension $d$, we have
$$a_d(G_I(M))+d\le \operatorname{red}(I).$$
Recall that for a maximal Cohen–Macaulay $A$-module $E \in \CMS(A)$, the thick subcategory $\operatorname{thick}(E)$ is defined as the intersection of all thick subcategories of $\CMS(A)$ that contain $E$. We now present a generalization of Trung’s result in our setting.
\medskip

\noindent
\textbf{Theorem V.} (Theorem \ref{Trunggen}) \phantomsection\label{Theorem V} \textit{Let $(A,\m,k)$ be a non-regular Gorenstein local ring of dimension $d\ge 1
$ and $I$ an $\m$-primary ideal. Let $E$ be a MCM $A$-module such that $r_\infty^{I,A}(E)\le d-q$ for some integer $q\ge 1$. Then there exists an integer $\eta$ depending only on $I$ and $A$, such that for any MCM $A$-module $M\in \operatorname{thick}(E)$ we have  $$\operatorname{max}\{a_d(G_I(M)),\ldots, a_{d-q+1}(G_I(M))\}\le \eta.$$
}

We also examine the case when $r_1^{I,N}(\Omega^dk)=-1$, that is, the function $n\mapsto \Tor_1^A(\Omega^dk,N/I^{n+1}N) $ is given by the zero polynomial for $n\gg 0$. Before we state an application we need the following notation. The graded ring $G_I(A)$ has a unique graded maximal ideal $\mathfrak{M}_G=\m /I \oplus_{n\ge 1}I^n/I^{n+1}$. For an $A$-module $X$, set $\depth G_I(X)=\grade(\mathfrak{M}_G,G_I(X))$. We conclue this paper with the following theorem.
\medskip

\noindent
\textbf{Theorem VI.} (Theorem \ref{depthdueto-1}) [With hypothesis as in \ref{setup0}] \phantomsection\label{Theorem VI} \textit{If $\pdim_A(N/I^{n+1}N)<\infty$ for all $n\ge 0$, then for any MCM $A$-module $M$, we have $\depth G_I(M\otimes N)\ge \depth G_I(N)$.
}
\medskip

Finally, we explore additional cases in which the inequalities in \hyperref[Theorem I]{Theorem I} become equalities. The failure of the conclusion of \hyperref[Theorem III]{Theorem III} for arbitrary $\m$-primary ideals suggests that further assumptions are required to ensure the desired equalities.

 \noindent
 \begin{definition} \normalfont We now recall two definitions.
   \begin{enumerate} [\rm (i)]
       \item An $A$-module $U$ is said to be a \textit{Test module}, if for all $A$-modules $V$ with $\Tor^A_{i\gg 0}(U,V)=0$ have finite projective dimension.
       \item A Noetherian ring $R$ is said to satisfy the \textit{Tor-vanishing property} if for finite $R$-modules $X,Y$ with $\Tor^R_{i\gg 0}(X,Y)=0$ implies that either $X$ or $Y$ has finite projective dimension.
   \end{enumerate}
    We note that test modules are abundant; see \ref{egtestmodule}. Also, there exist large classes of rings that satisfy the Tor-vanishing property; see \ref{egtvp}.
 \end{definition}

 \noindent
\textbf{Theorem VII.} (Theorem \ref{infinitei}) [With hypothesis be as in \ref{setup0}] \phantomsection\label{Theorem VII} \textit{Assume further that the residue field $k$ is uncountable. Let $M$ be a non-free MCM $A$-module, which is also a test module. Then there exists infinitely many integers $i$ such that $r_i^{I,N}(M)=r_1^{I,N}(\Omega^dk)$.}

\medskip

\noindent More generally, if $A$ satisfies Tor-vanishing property, then every non-free MCM $A$-module  is a test module, therefore every non-free MCM $A$-module fulfills the conclusion of the preceding theorem.
\medskip

\textit{Techniques used to prove the results:} Most of the results in this paper are established by induction on $t=\dim N$. The case $t=1$ serves as the foundation of our approach, while the higher-dimensional cases are handled using the well-known technique of common superficial and filter-regular elements. For the base case, we make use the notion of the cosyzygy of a MCM module over a Gorenstein rings. We study the associated graded module $G_I(M)$ via the $\mathcal{R}(I)$-module $L^I(M)=\bigoplus _{n\ge 0}M/I^{n+1}M$. Recall a graded module $T$ over a graded ring $R$ is said to be \textit{*-Artinian} if every descending chain of graded submodules of $T$ terminates. The second author in his paper \cite{Part1} proved that $L^I(M)$ is a module over the Rees algebra $\mathcal{R}(I)$ (not finitely generated)
 and the local cohomology modules $H^i_{\mathfrak{M}}(L^I(M))$ is *-Artinian for $0\le j \le \depth  M-1$, where $\mathfrak{M}=\m \oplus \mathcal{R}(I)_+$, the *maximal ideal of $\mathcal{R}(I)$ (see \cite[Proposition 4.4]{Part1}). For some properties of *-Aritinian modules we refer to \cite[1.10]{Part1}. We use the short exact sequence of $\mathcal{R}(I)$-modules
$$ 0\longrightarrow G_I(M)\longrightarrow L^I(M)\longrightarrow L^I(M)(-1)\longrightarrow 0$$
to have control over the numbers  $a_j(G_I(M))$ for $j=0,\ldots,\depth M-1$.
\medskip

Here is an overview of the contents of the paper. In \hyperref[section 2]{section 2} we discuss a few preliminaries. In \hyperref[section 3]{section 3} and \hyperref[section 4]{section 4}, we prove the \hyperref[Theorem I]{Theorem I} and \hyperref[Theorem II]{Theorem II} respectively. In \hyperref[section 5]{section 5}, we consider the case when $I=\m$ and we prove \hyperref[Theorem III]{Theorem III}. We also provide an example showing that the conclusions of \hyperref[Theorem III]{Theorem III} need not hold for arbitrary $\m$-primary ideals.  In \hyperref[section 6]{section 6}, we investigate various thick subcategories of the stable category $\CMS(A)$, examine their properties, and present a proof of \hyperref[Theorem IV]{Theorem IV}. \hyperref[section 7]{Section 7} begins with a discussion of the $\mathcal{R}(I)$-modules $L^I(M\otimes N)$ and $G_I(M\otimes N)$ and the relation between the graded local cohomology modules $H^{t-1}_{\mathfrak{M}}(L^I(M\otimes N)),H^{t}_{\mathfrak{M}}(G_I(M\otimes N))$ and $H^{t}_{\mathfrak{M}}(G_I(N))$. We also prove the \hyperref[Theorem V]{Theorem V} and \hyperref[Theorem VI]{Theorem VI}. The \hyperref[final section]{final section} contains the definition of a test module and the Tor-vanishing property, followed by several examples and the proof of \hyperref[Theorem VII]{Theorem VII}.

\section{Notation and Preliminaries} \phantomsection \label{section 2}
In this section we introduce some notation and discuss a few preliminaries which will be used in this paper.

\begin{setup} \normalfont \label{notation}
 Let $(A,\m,k)$  be Gorenstein local ring, which is not regular. Assume $N$ is a perfect  module of dimension $t\geq 1$ and $I$ is an ideal of definition of $N$. For a non-free MCM $A$-module $M$ and a positive integer $i$, we set \begin{align*}
   r_{i}^{I,N}(M)=\deg (n\mapsto \lambda(\Tor_{i}^{A}(M,N/I^nN))) ,\\
   s^{i}_{I,N}(M)=\deg (n\mapsto \lambda(\Ext^{i}_{A}(M,N/I^nN))).
\end{align*}
When $I=\m$ we write $r_{i}^{N}(M)$ and $s^{i}_{N}(M)$ instead of $r_{i}^{\m,N}(M)$ and $s^{i}_{\m,N}(M)$, respectively. For $i\ge0$, let $\Omega^iM$ denote the ${\operatorname{i^th}}$-syzygy of $M$. For a finitely generated $A$-module $X$, we set $\mu(X)=\dim_k(X/\m X)$.
\end{setup}

\begin{remark} \normalfont \label{etpi=1}
     Let $i\geq 2$ be an integer. Then we have
\begin{align*}
    r_{i}^{I,N}(M)=r_{i-1}^{I,N}(\Omega M), \\
    s^{i}_{I,N}(M)=s^{i-1}_{I,N}(\Omega M).
\end{align*}
Therefore, often we can assume $i=1$.
 \end{remark}

\begin{point}\normalfont \label{basechange} If the residue field $k$ of $A$ is finite,
we make use of the following faithfully  flat extension $$  (A,\m,k) \longrightarrow (A[X]_{\m A[x]}, \m A[X]_{\m A[x]},k(X))$$ to assume the residue field of $A$ is infinite. We also consider the extension $A \longrightarrow A[[X]]_{\m A[[X]]}$ in order to obtain an uncountable residue field.
\end{point}

\begin{point}\normalfont\normalfont \textbf{Cosyzygy:} \label{cosyzygy}
  Let $M$ be a non-free MCM over a Gorenstein local ring $(A,\m)$. We will define the \textit{cosyzygy} of $M$. Set $(-)^{*}=\Hom_{A}(-,A)$. Consider a minimal presentation
$
    0\longrightarrow \Omega (M^{*}) \longrightarrow A^{\tau}\longrightarrow M^* \longrightarrow 0
$ of $M^{*}$, where $\tau =\operatorname{type} (M)$. Since $M$ is MCM and $A$ is Gorenstein, therefore we have  $M^{**} \cong M$ and $\Ext_{A}^{i}(M^*,A)=0$ for $ i\geq 1$ (see \cite[Theorem 3.3.10]{Bruns and Herzog}). Hence applying the contra-variant functor $(-)^*$, we obtain the following exact sequence $ 0\longrightarrow M\longrightarrow A^{\tau} \longrightarrow (\Omega (M^{*}))^{*}\longrightarrow 0$. We set $\Omega^{-1}M:=(\Omega (M^{*}))^{*}$, it is called the $1^{\operatorname{st}}$-\textit{cosyzygy} of $M$. Recursively, we can define the ${\operatorname{i^th}}$-\textit{cosyzygy} of $M$ as $\Omega^{-i}(M):=\Omega^{-1}(\Omega^{-(i-1)}(M))$.
 Note that $\Omega^{i}M$ is non-free MCM for any integer $i$ and $\Omega \Omega^{-1}M\cong M$.
\end{point}

Next, we establish a few results that will be used in the subsequent discussion.

 \begin{lemma} \label{torext0}
   Let $(A,\m)$ be a $d$-dimensional Gorenstein local ring and $M$  a MCM $A$-module and $N$ a finite $A$-module with finite projective dimension (equivalently, finite injective dimension), then for any positive integer $i$, we have $   \Tor_{i}^{A}(M,N)=0=\Ext_{A}^{i}(M,N).$
\end{lemma}
\begin{proof} For $i\ge 1$, we have $  \Tor_{i}^{A}(M,N)=\Tor_{i+d}^{A}(\Omega^{-d}M,N)=0$. As $A$ is Gorenstein and $\operatorname{pdim}_{A}N <\infty$, we get $ \operatorname{idim}_{A}N < \infty$ (see \cite[Exercise 3.1.25] {Bruns and Herzog}). So for $i\ge 1$, we have $ \Ext_{A}^{i}(M,N) =\Ext_{A}^{i+d}(\Omega^{-d}M,N)=0.
 $
\end{proof}

\begin{proposition} \label{homtensorCM}
   Let $(A,\m)$ be Gorenstein local ring and $M$ a MCM $A$-module and $N$ a perfect module of dimension $t$. Then $M \otimes_{A} N$ and $\Hom_{A}(M,N)$ are Cohen-Macaulay $A$-module of dimension $t$.
\end{proposition}
\begin{proof}
    If $\dim A=0$, then the assertions are trivial. Let $\dim A\ge 1$. We note that if $M,N$ are nonzero then so is $M \otimes_{A} N$. Also note that $\Hom_{A}(M,N)=0$ if and only if $\ann_{A}M $ contains an $N$-regular element. Since $\operatorname{pdim}_AN< \infty$, so by new intersection theorem, an $N$-regular element is also $A$-regular (see \cite[Theorem 9.4.7 and Remark 9.4.8]{Bruns and Herzog}). As $M$ is MCM, it implies that $\depth(\ann_AM,A)=\operatorname{ht} (\ann_AM)=0$. So $\ann_A(M)$ does not contain an $N$-regular element and  hence $\Hom_{A}(M,N)\neq 0$.

    We prove the results by induction on $t=\dim N$. When $t=0$, there is nothing to prove. Assume $t\geq1$. Let $x\in \m$ be $N$-regular. Then $N/xN$ is a perfect module of dimension $t-1$. By \ref{torext0} we have
$\Tor_{1}^{A}(M,N/xN)=0=\Ext_{A}^{1}(M,N/xN)$. The exact sequence
$
    0\longrightarrow N \xlongrightarrow{x} N\longrightarrow N/xN\longrightarrow 0 $
yields the following two exact sequences
$$0\longrightarrow  M \otimes_{A} N \xlongrightarrow{x}  M \otimes_{A} N \longrightarrow M \otimes_{A} (N/xN)\longrightarrow 0,
$$
 $$0\longrightarrow  \Hom_{A}(M,N) \xlongrightarrow{x}  \Hom_{A}(M,N)\longrightarrow \Hom_{A}(M,N/xN)\longrightarrow 0,$$
which implies that $x$ is both $M \otimes_{A} N$ and $\Hom_{A}(M,N)$-regular and $$ \Hom_{A}(M,\frac{N}{xN})\cong \frac{\Hom_{A}(M,N)}{x\Hom_{A}(M,N)}.$$ Also $M \otimes_{A} N/xN$ and $\Hom_{A}(M,N/xN)$ are Cohen-Macaulay of dimension ($t-1$) by induction hypothesis. The result follows.
    \end{proof}

 We now recall the notion of the \textit{multiplicity} of a module.
 \begin{definition} \normalfont Let $(A,\m)$ be a local ring and $M$ a finite $A$-module of dimension $d$. Then the \textit{multiplicity} of $M$ is defined as $$e(M)=\lim_{n \rightarrow \infty}\frac{d!}{n^d}\lambda\left(\frac{M}{\m^{n+1}M}\right).$$

 \end{definition}
 \medskip

 We next recall a result of Theodorescu that will play a central role in our discussion.
     \begin{point}\normalfont Let $\mathcal{R}(I)=A[Iu]=\bigoplus_{n\ge 0} I^n u^n$ denote the Rees algebra of $I$, where $u$ is an indeterminate serving as a placeholder for the grading.
 Let $\mathcal{F}_N(I)=\bigoplus_{n\ge 0} I^nN/\m I^nN$ be the fiber-cone of $I$ with respect to $N$, considered as a module over the Rees algebra $\mathcal{R}(I)$ of $I.$ Let $l_N(I) = \dim \mathcal{F}_N(I)$, the analytic spread of $I$ with respect to $N$.
We note that $l_N(I) \le \dim N$ and if $I$ is an ideal of definition of $N$, then $l_N(I) = \dim N$.
\medskip
The following result is due to Theodorescu .
\end{point}

\begin{theorem} (\cite[Corollary 4]{E.Theodorescu}) \label{Theodorescu}
   Let $(A, \mathfrak{m})$ be a Noetherian local ring.
Let $M, N$ be finitely generated $A$-modules and let $I$ be an ideal in $A$.
Fix $i \ge 0$.

\begin{enumerate} [\rm (I)]
    \item Assume $\lambda(\operatorname{Tor}^A_i(M, N / I^n N))$ is finite for all $n \ge 1$. Then
    \begin{enumerate} [\rm (i)]
        \item The function $n \mapsto \lambda(\operatorname{Tor}^A_i(M, N / I^n N))$ is of polynomial type, say of degree $r_{i}^{I,N}(M)$.
        \item $r_{i}^{I,N}(M) \le \max\{\dim \operatorname{Tor}^A_i(M, N),\, l_N(I) - 1\}$.
        \item If $\dim \operatorname{Tor}^A_i(M, N) \ge l_N(I)$, then the inequality in $\operatorname{(ii)}$ becomes an equality.
    \end{enumerate}

    \item Assume $\lambda(\operatorname{Ext}^i_A(M, N / I^n N))$ is finite for all $n \ge 1$. Then
    \begin{enumerate} [\rm (i)]
        \item The function $n \mapsto \lambda(\operatorname{Ext}^i_A(M, N / I^n N))$ is of polynomial type, say of degree $ s^{i}_{I,N}(M)$.
        \item $ s^{i}_{I,N}(M) \le \max\{\dim \operatorname{Ext}^i_A(M, N),\, l_N(I) - 1\}$.
        \item  If $\dim \operatorname{Ext}^i_A(M, N) \ge l_N(I)$, then the inequality in $\operatorname{(ii)}$ becomes an equality.
    \end{enumerate}
\end{enumerate}
\end{theorem}

\begin{point}\normalfont \label{bound due to theodorscu}
     (With hypothesis as in 2.1)  For any integer $i\geq 1$, from \ref{torext0} we have $\Tor_i^A(M,N)=0=\Ext^i_A(M,N)$, therefore we obtain
     $$r_{i}^{I,N}(M)\leq t-1 \ \operatorname{and} \ s^{i}_{I,N}(M)\leq t-1.$$
Note that, if $\dim N=t=1$, then the functions $n \mapsto \Tor_i^A(M,N/I^{n+1}N)$ and $n \mapsto \Ext^i_A(M,N/I^{n+1}N)$ are eventually constant. We also examine the case where $t = \dim N \ge 2$, employing superficial and filter-regular elements.
\end{point}

\begin{point}\normalfont
  Assume $A$  is Gorenstein, $N$ is a perfect $A$-module and $I$ is an ideal of definition of $N$. Let $M$ be a MCM $A$-module. For $i \geq 1$ set
$$
L_i^{I,N}(M) = \bigoplus_{n \geq o} \text{Tor}_i^A(M, N/I^{n+1}N) \quad \text{and} \quad E_{I,N}^i(M) = \bigoplus_{n \geq 0} \text{Ext}^i_A(M, N/I^{n+1}N).
$$
\end{point}
The following result is proved in \cite[Proposition 9.5]{HypersurfaceII}
\begin{proposition} \label{LiEj}
  (With hypotheses as in 2.6) For $ i \geq 1 $, $ L_i^{I,N}(M)$ and $E_{I,N}^i(M)$ are finitely generated graded $ \mathcal{R}(I)$-modules.
\end{proposition}

We now recall the notions of superficial elements and filter-regular elements, which will be used in the subsequent discussion.

\begin{point}\normalfont
  An element $x \in I$ is $M$-superficial with respect to $I$ if there exists $c$ such that $(I^{n+1}M : x) \cap I^c M = I^n M$ for all $n\gg 0.$ Assume $\lambda(M/IM)$ is finite. If $\grade(I, M) > 0$ then it follows that $x$ is $M$-regular and $(I^{n+1}M : x) = I^n M $ for $n\gg 0$. If the residue field $k$ is infinite then $I$-superficial elements with respect to $M$ exist. In fact in this case there exists a non-empty open set $U_M$ in the Zariski topology of $I/\mathfrak{m}I$ such that if the image of $x$ is in $U_M$ then $x$ is $I$-superficial with respect to $M$.

   A sequence of elements $x_1,x_2,\ldots,x_s$ is called a superficial sequence for $M$ with respect to $I$ if $\bar{x_i}$ is superficial for $M/(x_1,\ldots,x_{i-1})M$ for $i=1,\ldots,s$.
\end{point}

\begin{point}\normalfont
     Let $E = \bigoplus_{n \ge 0} E_n$ be a finitely generated graded module over the Rees algebra $\mathcal{R} = A[It]$. Assume $E_n$ has finite length for all $n$. There exists $xt \in \mathcal{R}_1$ such that $xt$ is $E$-filter regular, that is, $(0 :_E xt)_n = 0 $ for $n\gg 0$. In fact in this case there exists a non-empty open set $U_E$ in the Zariski topology of $I/\mathfrak{m}I$ such that if the image of $x$ is in $U_E$ then $xt$ is $E$-filter regular.

      A sequence of elements $x_1,x_2,\ldots,x_s$ is called a filter-regular sequence for $E$  if $\bar{x_i}$ is filter-regular for $E/(x_1,\ldots,x_{i-1})E$ for $i=1,\ldots,s$.
\end{point}

We are now in a position to establish the following two exact sequences, which will be fundamental to the developments in the subsequent sections.

\begin{point}\normalfont \label{-1argu}
  Assume $t = \dim N \ge 2$ and $i,j$ are integers $\geq 1$ and assume $M$ is a MCM $A$-module. Also assume $A$ is Gorenstein and $N$ is a perfect $A$-module. Let $x \in I$ be such that it is $N \oplus (M \otimes N)$-superficial and $xt$ is $L_p^N(M)$, $E_q^N(M)$-filter regular for $p=i,i-1$ and $q=j,j+1$. We note that $(I^{n+1}N : x) = I^n N$ and $(I^{n+1}(M \otimes N) : x) = I^n (M \otimes N)$ for $n \gg 0$ (here we are using $M \otimes N$ is Cohen--Macaulay see \ref{homtensorCM}). We note that $\overline{N} = N/xN$ is a perfect $A$-module of dimension $t-1$. We have an exact sequence for $n \ge 1$
$$
0 \to \ker \alpha_n \to N/I^n N \xrightarrow{\alpha_n} N/I^{n+1} N \to \overline{N}/I^{n+1} \overline{N} \to 0,
$$
where $\alpha_n(a + I^n) = xa + I^{n+1}$. We note that $\ker \alpha_n = (I^{n+1}N : x)/I^n N = 0$ for $n \gg 0$. Thus for $n \gg 0$ we have an exact sequence
$$
\operatorname{Tor}_i^A(M, N/I^n N) \xrightarrow{\alpha_n^i} \operatorname{Tor}_i^A(M, N/I^{n+1} N) \to \operatorname{Tor}_i^A(M, \overline{N}/I^{n+1} \overline{N}) \to
$$
$$
\operatorname{Tor}_{i-1}^A(M, N/I^n N) \xrightarrow{\alpha_n^{i-1}} \operatorname{Tor}_{i-1}^A(M, N/I^{n+1} N)
$$
We note that $\ker \alpha_n^0 = (I^{n+1}(M \otimes N) : x)/I^n (M \otimes N) = 0$ for $n \gg 0$. Furthermore, the map $\alpha_n^p$ is the $n^{th}$-component of the multiplication map by $xt \in \mathcal{R}_1$ on $L_p^N(M)$. As $xt$ is $L_p^N(M)$-filter regular for $p=i,i-1$, it follows that $\ker \alpha_n^p = 0$ for $n \gg 0$. Thus for $n \gg 0$ we have an exact sequence
\end{point}

\begin{point}\normalfont \label{F1tor} (under the hypothesis stated above)
 $$
0 \to \operatorname{Tor}_i^A(M, N/I^n N) \xrightarrow{\alpha_n^i} \operatorname{Tor}_i^A(M, N/I^{n+1} N) \to \operatorname{Tor}_i^A(M, \overline{N}/I^{n+1} \overline{N}) \to 0.
$$
\end{point}

Similarly as $xt$ is also $E_q^N(M)$-filter regular for $q=j,j+1$, we have an exact sequence for $n\gg 0$
\begin{point}\normalfont \label{F2ext} (under the hypothesis stated above)
 $$
0 \to \operatorname{Ext}_A^j(M, N/I^n N) \xrightarrow{\beta_n^j} \operatorname{Ext}_A^j(M, N/I^{n+1} N) \to \operatorname{Ext}_A^j(M, \overline{N}/I^{n+1} \overline{N}) \to 0.
$$
\end{point}

Next we recall several results from \cite{Part1} that will be used in the subsequent sections.

\begin{point}\normalfont \label{superficialstarting}
  Let $J$ be an ideal of $A$ and $T$ a finite $A$-module. Assume grade($J,T)>0$. Set $$\widetilde{JT}=\bigcup _{i\ge 0}(J^{i+1}T:_T J^i).$$  It was proved in \cite{Ratliff-Rush} that $\widetilde{J^nT}=J^nT$ for all $n\gg 0$. Set
   $$\rho^J(T):=\operatorname{min}\{ i|\widetilde{J^nT}=J^nT \ \operatorname{for} \ \operatorname{all} \ n\ge i\}.$$
   If $x\in J$ is $T$-superficial with respect to $J$, then we get $(J^{n+1}T:_Tx)=J^nT$ for all $n\gg 0$. Set
   $$\rho^J(x,T):=\operatorname{min}\{ i|(J^{n+1}T:_Tx)=J^nT\ \operatorname{for} \ \operatorname{all} \ n\ge i\} .$$
Then we have $\rho^J(x,T)=\rho^J(T)$, thus $\rho^J(x,T)$ is independent of superficial elements (for a proof see  \cite[Corollary 2.7]{Part1}).

Let $M$ be a non-free MCM $A$-module. Then (with the hypothesis as in \ref{notation}), it was proved in \cite[Proposition 4.7]{Part1} that
$$H^0_\mathfrak{M}(L^I(M\otimes N))=\bigoplus_{i=0}^{\rho^I(M\otimes N)-1} \frac{\widetilde{I^{i+1}(M\otimes N)}}{I^{i+1}(M\otimes N)}.$$
If $x\in I$ is $M\otimes N$-superficial with respect to $I$, then for all $n\ge \operatorname{end}(H^0_\mathfrak{M}(L^I(M\otimes N)))$, we have
$$ I^{n+1}(M\otimes N):x=I^n(M\otimes N).$$

\end{point}

\begin{point}\normalfont
     Let $(R_0,\m_0)$ be a local ring and  $R=\bigoplus_{n\ge 0}R_n $ be a standard graded $R_0$-algebra, that is, $R$ is generated over $R_0$ by finitely many elements of degree 1. Let $L$ be a (not necessarily finitely generated) graded $R$-module. Define $\operatorname{end}(L)=\sup \{n\in \Z|L_n\ne 0\}$. Recall if $L$ is a *-Artinian $R$-module then $\operatorname{end}(L)<\infty$ (see \cite[Lemma 1.10]{Part1}). Set $R_+=\bigoplus_{n\ge1}R_n$. If $E$ is a finitely generated $R$-module, then for each $i\ge 0$ we have $H^i_{R_+}(E)_n=0$ for all $n\gg 0$, which implies $\operatorname{end}(H^i_{R_+}(E))<\infty$ for all $i\ge 0$ (cf. \cite[Theorem 15.1.5]{Brodman-Sharp}).
\end{point}

\begin{point}\normalfont \label{b_j<a_j+1}
  Let $X$ be a finite $A$-module. Set $L^I(X)=\bigoplus _{n\ge 0}X/I^{n+1}X$, and for an integer $i\ge 1$, set $L_i^{I,N}(X)=\bigoplus _{n\ge 0}\Tor_i^A(X,N/I^{n+1}N)$. We have proved in \ref{LiEj} that the modules $L_i^{I,N}(X)$ are finitely generated  graded  $\mathcal{R}(I) $-modules for $i\ge 1$. Note that $L^I(X)$ is \textbf{not} a finitely generated $\mathcal{R}(I)$-module.

  Set $\mathfrak{M}=\m \oplus \mathcal{R}(I)_+$. It was proved in \cite[Proposition 4.4]{Part1} that $H^i_{\mathfrak{M}}(L^I(X))$ is *-Artinian for $0\le j \le \operatorname{depth}X-1$. For $j=0,\ldots ,\operatorname{depth}X-1$, we define the following invariants $$b_j^I(X):=\operatorname{end}(H^j_\mathfrak{M}(L^I(X)).$$
  Since $H^j_\mathfrak{M}(G_I(X))$ is *-Artinian we can define
  $$a_j^*(G_I(X)):=\operatorname{end}(H^j_\mathfrak{M}(G_I(X))) \ \operatorname{for} \ j\ge 0.$$
Recall $a_j(G_I(X))=\operatorname{end}(H^j_{G_I(A)_+}(G_I(X)))$. Assume $I$ is an ideal of definition of $X$, then we have
$$
   a_j^*(G_I(X))=a_j(G_I(X)).
$$
The natural maps $0\longrightarrow I^nX/I^{n+1}X\longrightarrow X/I^{n+1}X\longrightarrow X/I^nX\longrightarrow 0$ induce an exact sequence of $\mathcal{R}(I)$-modules
$$
    0\longrightarrow G_I(X)\longrightarrow L^I(X)\longrightarrow L^I(X)(-1)\longrightarrow 0.
$$
So we have the following exact sequence
$$0=H^j_\mathfrak{M}(L^I(X))_{b^I_j(X)+1}\longrightarrow H^j_\mathfrak{M}(L^I(X))_{b^I_j(X)}\longrightarrow H^{j+1}_\mathfrak{M}(G_I(X))_{b^I_j(X)+1}, $$
therefore we obtain $H^{j+1}_\mathfrak{M}(G_I(X))_{b^I_j(X)+1} \neq 0$, and so we have
$$
    b_j^I(X)\le a^*_{j+1}(G_I(X))-1=a_{j+1}(G_I(X))-1 \ \operatorname{for}\ 0\le j\le \operatorname{depth}X-1.
$$
\end{point}

We now recall the definition of the complexity of a module, along with some elementary properties.

\begin{point}\normalfont
    The notion of complexity was introduced by Avramov in \cite{cxll}. For a finitely generated module $M$ over a local ring $(A,\m)$, the \textit{complexity} of $M$ is defined as
    $$
\mathrm{cx}_AM = \inf \left\{ d \in \mathbb{N} \,\middle|\,
\begin{array}{l}
\text{there exists a polynomial } f(X) \text{ of degree } d - 1 \\
\text{such that } \beta_n^A(M) \le f(n) \text{ for } n \ge 0
\end{array}
\right\},
$$
where $\beta_n^A(M)=\dim_k\Tor_n^A(M,k)$ is the $n^{\operatorname{th}}$-Betti number of $M$ over $A$. Note that $\cx_A M=0$ if and only if $\pdim_AM<\infty$. Furthermore $\cx_A M\le 1$ if and only if $M$ has bounded Betti numbers. If $A$ is a local complete intersection, then $\cx_AM\le \codim A$ (see \cite[Corollary 4.2]{cxcodim}). If $M$ is a MCM $A$-module with no free summand and $\cx_A M = 1$, then $M$ is $2$-periodic; that is, $M \cong \Omega^2 M$ (see \cite[Theorem 9.2.1]{cx1}).

\end{point}

We also recall the notion of \textit{multiplicity} of a module.
\begin{point}
    \normalfont Let $(A,\m)$be a local ring. The \textit{multiplicity} of a finite $A$-module $M$ of dimesnion $d$ is defined as
    $$e(M)=\operatorname{lim}_{n\rightarrow \infty}\frac{d!}{n^d}\lambda\left(\frac{M}{\m^nM}\right).$$
    Suppose we have an exact sequence $0\longrightarrow M'\longrightarrow M\longrightarrow M''\longrightarrow 0$, where all three modules have same dimension. Then  $e(M)=e(M')+e(M'')$; see \cite[Corollary 4.7.7]{Bruns and Herzog}. Note that $e(M)>0$, provided $ M \neq 0$. Also, if $x\in \m$ is $M$-superficial with respect to $\m$, then we have $e(M)=e(M/xM)$.
\end{point}

We conclude this section by briefly discussing the triangulated category $\underline{\operatorname{CM}}(A)$ and some of its properties.

    \begin{point} \textbf{The Triangulated category} \normalfont $\underline{\operatorname{CM}}(A)$.\\
  We use \cite{Neeman} for notation on triangulated categories.  Let $(A,\m)$ be a Goresntein local ring. Let $\operatorname{CM}(A)$ denote the full subcategory of finitely generated MCM $A$-modules and let $\operatorname{\underline{CM}(A)}$ denote the stable category of MCM $A$-modules.  Recall that objects in $\operatorname{\underline{CM}(A)}$ are same as objects in $\operatorname{CM}(A)$. However the set of morphisms of $\operatorname{\underline{Hom_A}(M,M')}$ between two MCM $A$-modules $M$ and $M'$ is $=\operatorname{Hom}_A(M,M')/P(M,M')$, where $P(M,M')$ is the set of $A$-linear maps from $M$ to $M'$ which factor through a finitely generated free module. It is well-known that $\operatorname{\underline{CM}(A)}$ is a triangulated category with translation functor $\Omega^{-1}$ (see \cite{CM(A)}). Also recall that an object $M$ is zero in $\operatorname{\underline{CM}(A)}$ if and only if it is free considered as an $A$-module. Furthermore $M\cong M'$ in  $\operatorname{\underline{CM}(A)}$  if and only if there exists finitely generated free modules $F,G$ with $M\oplus F\cong M'\oplus G$ as $A$-modules. A distinguished triangle in $\operatorname{\underline{CM}(A)}$ has the form $X\longrightarrow Y\longrightarrow Z\longrightarrow \Omega^{-1}X$, (for more detailed discussion see \cite[4.7]{CM(A)} and \cite[2.11]{StableCategory_Goresntein}). For the definition of thick subcategories see \cite{Neeman}.
\end{point}

We now state a property of thick subcategories of $\CMS(A)$.

\begin{lemma} \label{thichkcx1}
    Let $(A,\m)$ be a local non-regular, complete intersection ring. Then any nonzero thick subcategory of $\CMS(A)$ contains a MCM $A$-module $E$ with  $\cx_AE=1$.
\end{lemma}
\begin{proof}
    Let $T$ be a nonzero thick subcategory of $\CMS(A)$. Pick a nonzero non-free MCM $A$-module $M$ in $T$, therefore $\operatorname{cx}_AM\ge 1$. If $\operatorname{cx}_AM=1$, we have nothing to prove. Assume $\operatorname{cx}_AM\ge 2$, so there exist integers $n_0$ and $r$ such that we have an exact sequence $0\longrightarrow C \longrightarrow \Omega^{n_{0}+2r}M\longrightarrow \Omega^{n_0}M\longrightarrow 0$, with $\operatorname{cx}_AC=\operatorname{cx}_AM-1$ (this is essentially contained in \cite[Theorem 3.8]{StableCategory_Goresntein}). Since every short exact sequence in $\CMS(A)$ induces an exact triangle, so we obtain an exact triangle $C\longrightarrow \Omega^{n_{0}+2r}M\longrightarrow \Omega^{n_0}M\longrightarrow \Omega^{-1}C$ in $\CMS(A)$. As $T$ is closed under triangles, it follows that $C\in T$. Hence, by induction, so we obtain a MCM $A$-module $E\in T$ such that $\operatorname{cx}_AE=1$.
\end{proof}

\section{Bounds and Estimates} \phantomsection \label{section 3}
In this section, we aim to investigate the bound established in \ref{bound due to theodorscu}. We adopt the convention that the degree of the zero polynomial is $-1$.
\medskip

 \begin{setup}\normalfont \label{setup}
   Throughout this section $(A,\m,k)$ is a non-regular Gorenstein local ring of dimension $d\ge 1$ which is not regular and $N$ is a perfect module of dimension $t\geq 1$ and $I$ is an ideal of definition of $N$. For a non-free MCM $A$-module $M$ and a positive integer $i$, we set  $r_{i}^{I,N}(M)$, $s^{i}_{I,N}(M)$, $r_i^N(M)$ and $s^i_N(M)$ as in \ref{notation}. Using \ref{basechange} we can assume the residue field $k$ is infinite.

 \end{setup}
We now state a lemma that is essential for proving our bound.
 \begin{lemma} \label{-1case} (With hypothesis as in \ref{setup})
     Let $i\geq 1$ be an integer. If either $r_{i}^{I,N}(\Omega^dk)=-1$ or $s^{i}_{I,N}(\Omega^dk)=-1$, then for any integer $j\geq 1$  and MCM $A$-module $M$, we have
\begin{equation*}
     r_{j}^{I,N}(M)=-1=s^{j}_{I,N}(M).
\end{equation*}
 \end{lemma}
\begin{proof}
    Assume $r_{i}^{I,N}(\Omega^dk)=-1$. Then for $n\gg 0$ we have \begin{equation*}
        \operatorname{Tor}_{i}^{A}(\Omega^dk,N/I^nN)=\operatorname{Tor}_{i+d}^{A}(k,N/I^nN)=\beta_{i+d}(N/I^nN)=0.
    \end{equation*}
    Therefore, $\operatorname{pdim}_A(N/I^nN)< \infty$ for $n\gg 0$.

    If $s^{i}_{I,N}(\Omega^dk)=-1$, then for $n\gg 0$ we have
    \begin{equation*}
        \operatorname{Ext}_{A}^{i}(\Omega^dk,N/I^nN)=\operatorname{Ext}^{i+d}_{A}(k,N/I^nN)=\mu ^{i+d}(\m,N/I^nN)=0.
    \end{equation*}
    Therefore, $\operatorname{idim}_{A}(N/I^nN)=0$ for $n\gg 0$ (see \cite[Exercise 3.5.12]{Bruns and Herzog}), and hence $\operatorname{pdim}_A(N/I^nN)< \infty$ for $n\gg 0$. Thus, in either case, the conclusion follows from \ref{torext0}.
\end{proof}
We now establish our bound.

\begin{theorem} \label{abound} (With hypothesis as in \ref{setup}) Let $M$ be a non-free MCM $A$-module. Let $i\ge 1$ be an integer. Then we have
\begin{enumerate} [\rm (i)]
    \item $ r_i^{I,N}(\Omega^dk)=r_1^{I,N}(\Omega^dk)$,
    \item $s^i_{I,N}(\Omega^dk)=s^1_{I,N}(\Omega^dk),$
    \item $r_1^{I,N}(\Omega^dk)=s^1_{I,N}(\Omega^dk)$,
    \item $r_i^{I,N}(M)\le r_1^{I,N}(\Omega^dk)\le t-1$,
    \item $s^i_{I,N}(M)\le s^1_{I,N}(\Omega^dk)\le t-1$.
\end{enumerate}
\end{theorem}
\begin{proof}(i) We write $r_j(N)$ to denote the value $r_{j}^{I,N}(\Omega^dk)$. We apply induction on $t=\operatorname{dim} N$. Suppose $t=1$, by \ref{bound due to theodorscu} $r_{i}(N)\leq 0$. Using \ref{-1case} we obtain, $r_i(N)=-1$ if and only if  $r_{1}(N)=-1$, hence for $t=1$ the result holds. Assume $t\geq 2$. By 3.3. we can assume both $r_i(N)$ and $r_1(N)$ are $\geq 0$. Let $x \in I$ be such that it is $N \oplus (\Omega^dk ) \otimes N)$-superficial and $xt$ is $L_p^N( \Omega^dk)$, -filter regular for $p=1,i,i-1$. Using the short exact sequence in \ref{F1tor} , we obtain  $r_i(N/xN)=r_i(N)-1$ and  $r_1(N/xN)=r_1(N)-1$. Note that $N/xN$ is a perfect module of dimension $t-1$, hence by induction hypothesis the result follows.

 (ii) is proved in similar lines as in (i).

(iii) We write $r_1(N)$ and $s^1(N)$ to denote $r_{1}^{I,N}(\Omega^dk)$ and $s^{1}_{I,N}(\Omega^dk)$ respectively. We apply induction on $t=\dim N$. Suppose $t=1$, then by \ref{bound due to theodorscu} we obtain $r_1(N)\leq 0$, $s^1(N)\leq 0$. By \ref{-1case} $r_1(N)=-1$ if and only if $s^1(N)=-1$, so the result holds for $t=1$. Assume $t\ge 2$. Using \ref{-1case} we can assume none of $r_1(N)$ or $s^1(N)$ is $-1$. Let $x \in I$ be such that it is $N \oplus (\Omega^dk ) \otimes N)$-superficial and $xt$ is $L_1^N( \Omega^dk)$, $E_q^N( \Omega^dk)$-filter regular for $q=1,2$. Using the short exact sequences in \ref{F1tor} and \ref{F2ext}, we obtain $r_1(N/xN)=r_1(N)-1$ and $s^1(N/xN)=s^1(N)-1$. Since $N/xN $ is a perfect of dimension $t-1$, hence by induction hypothesis the result follows.

(iv) We apply induction on $t=\dim N$. The $t=1$ case follows from \ref{bound due to theodorscu}  and \ref{-1case} . Assume $t\ge 2$. If  $r_1^{I,N}(\Omega^dk)=-1$, then  $r_i^{I,N}(M)=-1$ (see \ref{-1case}). If $r_i^{I,N}(M)=-1$,  the assertion is trivial. Therefore we can assume none of  $r_1^{I,N}(\Omega^dk)$ and $r_i^{I,N}(M)$ is $-1$. Let $x \in I$ be such that it is $N \oplus ((\Omega^dk \oplus M) \otimes N)$-superficial and $xt$ is $L_p^N(M \oplus \Omega^dk)$-filter regular for $p=1,i,i-1$. Using the short exact sequence in \ref{F1tor}, we obtain \begin{align*}
    r_1^{I,N/xN}(M)=r_1^{I,N}(M)-1, \\
    r_1^{I,N/xN}(\Omega^dk)=r_1^{I,N}(\Omega^dk)-1.
\end{align*}
  Since $N/xN$ is perfect of dimension $t-1$, so by induction hypothesis the result follows.

  (v) is proved in similar lines as in (iv).
\end{proof}

\section{Proof of Theorem II.} \phantomsection \label{section 4} In this section we give a proof of Theorem II. Before proceeding, we establish a few preliminary results that will be essential for the proof of our theorem.
\begin{lemma} \label{extomega}
    Let $(A,\m)$ be a $d$-dimensional Gorenstein local ring and $M$ a finite $A$-module of depth $s$ and dimension $r$. Let $i\ge 1$ be an integer such that $\{i,i+1\}\cap [d-r,d-s]=\emptyset$, then for any $A$-module $N$ (not necessarily finitely generated),  we have $\Ext^i_A(M,N)\cong \Ext^i_A(\Omega M,\Omega N)$.
\end{lemma}
\begin{proof}
  Apply $\Hom_A(-,\Omega N)$ to the exact sequence $0\longrightarrow \Omega M \longrightarrow A^{\mu(M)}\longrightarrow M\longrightarrow 0$ to obtain
  $
    \Ext^i_A(\Omega M,\Omega N)\cong \Ext^{i+1}_A(M,\Omega N).
  $
  Next, we apply $\Hom_A(M,-)$ to the exact sequence $0\longrightarrow \Omega N \longrightarrow G \longrightarrow N\longrightarrow 0$, where $G$ is free, to obtain the exact sequence $$\Ext^i_A(M,G)\longrightarrow \Ext^i_A(M,N)\longrightarrow \Ext^{i+1}_A(M,\Omega N)\longrightarrow \Ext^{i+1}_A(M,G).$$ By our assumption on $i$, we have $\Ext^i_A(M,G)=0=\Ext^{i+1}_A(M,G)$ (see \cite[Corollary 3.5.11]{Bruns and Herzog}), therefore we get
  $
      \Ext^i_A(M,N)\cong \Ext^{i+1}_A(M,\Omega N).
  $
  The result follows.
\end{proof}

\begin{lemma}   \label{Takahashi}
 Let $(A,\m,k)$ be a $d$-dimensional non-regular Gorenstein local ring and $E$  a MCM $A$-module. Suppose for an integer $n$, if we have $\Tor_1^A(\Omega^{n}\Omega^dk,E)=0$, then $E$ is free.
\end{lemma}
\begin{proof} The assertion is trivial for $n\ge 0$. Assume $n<0$. There are isomorphisms in  the derived category of $A$:  $\mathbf{R}\Hom_A(X\otimes _A^{\mathbf{L}}M,A)\cong \mathbf{R}\Hom_A(X,M^*)$ and $\mathbf{R}\Hom_A(\mathbf{R}\Hom_A(X,M^*),A)\cong X \otimes _A^{\mathbf{L}} \mathbf{R}\Hom_A(M^*,A)\cong X \otimes _A^{\mathbf{L}} M$ (see \cite[A.4.21 and A.4.24]{spectralsequence}). These gives rise to a  spectral sequence
    $$
        E_2^{p,q}=\Ext^p_A(\Tor_q^A(\Omega^{n}\Omega^dk,E),A)\implies H^{p+q}=\Ext^{p+q}_A(\Omega^{n}\Omega^dk,E^*).
$$
    Since $\Omega^{n}\Omega^dk$ is locally free on the punctured spectrum, so $\lambda(\Tor_q^A(\Omega^{-1}\Omega^dk,E))< \infty$ for $q> 0$. So by \cite[Corollary 3.5.11]{Bruns and Herzog}, we have $\Ext^p_A(\Tor_q^A(\Omega^{n}\Omega^dk,E),A)=0$ for $p\neq d$ and $q\neq 0$. So the above spectral sequence collapses at $2^{\operatorname{nd} }$-page. By hypothesis, $\Tor_1^A(\Omega^{n}\Omega^dk,E)=0$, thus  we obtain
    $\Ext^{d+1}_A(\Omega^{n}\Omega^dk,E^*)=0$. By repeatedly applying \ref{extomega}, we obtain (as $n<0$)
    $$0=\Ext^{d+1}_A(\Omega^{n}\Omega^dk,E^*)=\Ext^{d+1}_A(\Omega^{-n}\Omega^{n}\Omega^dk,\Omega^{-n}E^*)=\Ext^{d+1}_A(\Omega^dk,\Omega^{-n}E^*).$$
   So we have $\Ext^{2d+1}_A(k,\Omega^{-n}E^*)=0$, which implies that $\Omega^{-n}E^*$ is free, and thus $E^*$ is free and so is $E^{**}\cong E$.
\end{proof}
\begin{proposition} \label{omega^nomega^dk}
   Let $(A,\m,k)$ be a $d$-dimensional non-regular Gorenstein local ring, $E$  a finite $A$-module. Suppose for some integer $n$, one of the following two conditions holds:
   \begin{enumerate}  [\rm (i)]
       \item $\Tor_1^A(\Omega^{n}\Omega^dk,E)=0$,
       \item $\Ext^1_A(\Omega^{n}\Omega^dk,E)=0$,
   \end{enumerate}
   then we have $\pdim_A(E)<\infty$. \end{proposition}
\begin{proof} The assertions are trivial for $n\ge 0$. We assume $n<0$.\\
   (i) By a maximal Cohen-Macaulay approximation of $E$, we have an exact sequence
    $$0\longrightarrow Y\longrightarrow X\longrightarrow E\longrightarrow0,$$
    where $X$ is MCM and $\pdim_A(Y)<\infty$ (see \cite[Theorem A]{Auslander Buchweitz}). Tensoring this exact sequence with $\Omega^{n}\Omega^dk$, we obtain
    $$\Tor_1^A(\Omega^{n}\Omega^dk,Y)\longrightarrow \Tor_1^A(\Omega^{n}\Omega^dk,X)\longrightarrow \Tor_1^A(\Omega^{n}\Omega^dk,E).$$
     As $\pdim_A(Y)<\infty$, so by \ref{torext0} we get  $\Tor_1^A(\Omega^{n}\Omega^dk,Y)=0$ which implies $\Tor_1^A(\Omega^{n}\Omega^dk,X)=0$. Since $X$ is MCM, so by \ref{Takahashi} we have $X$ is free and hence $\pdim_A(E)<\infty$.

     (ii) We apply \ref{extomega} repeatedly to obtain (as $n<0$)
     $$0=\Ext^1_A(\Omega^{n}\Omega^dk,E)=\Ext^1_A(\Omega^{-n}\Omega^{n}\Omega^dk,\Omega^{-n}E)=\Ext^1_A(k,\Omega^{-n}E),$$
     so we have $\pdim_A(\Omega^{-n}E)<\infty$, and thus $\pdim_A(E)<\infty$.
\end{proof}
With the preceding results in place, we can now prove Theorem II.
\begin{theorem} \label{Omegadk=Omegan+dk} (with hypothesis as in \ref{setup}) Let $n$ be an integer, then we have
\begin{enumerate}  [\rm (i)]
\item $r_1^{I,N}(\Omega^dk)=r_1^{I,N}(\Omega^n\Omega^dk)$,
\item $s^1_{I,N}(\Omega^n\Omega^dk)=s^1_{I,N}(\Omega^dk).$
\end{enumerate}
\end{theorem}
\begin{proof} (i) If $n\ge 0$, the assertions follow from \ref{abound}. Assume $n<0$. We claim that \begin{equation*}
       r_1^{I,N}(\Omega^dk)=-1 \ \operatorname{iff} \  r_1^{I,N} (\Omega^n\Omega^dk)=-1.
    \end{equation*}
    The only if part directly follows from \ref{-1case}. If $r_1^{I,N} (\Omega^n\Omega^dk)=-1$, then $\Tor_1^A(\Omega^n\Omega^dk,N/I^nN)=0$ for $n\gg 0$ and so by \ref{omega^nomega^dk} we have $\pdim_A(N/I^nN)<\infty$ for $n\gg 0$. Thus by \ref{torext0} we have $\Tor_1^A(\Omega^dk,N/I^nN)=0$ for $n\gg 0$, which implies $ r_1^{I,N}(\Omega^dk)=-1$. \\
    We apply induction on $t=\dim N$. The case t=1 follows from \ref{bound due to theodorscu} together with the above argument. Assume $t\ge 2$.  we can assume that none of $r_1^{I,N}(\Omega^dk)$ and $r_1^{I,N}(\Omega^n\Omega^dk)$ is $-1$. Let $x\in I$ be such that it is $N\oplus ((\Omega^dk \oplus\Omega^n\Omega^dk)\otimes N)$-superficial and $xt\in A[It]_1$ is $L_1^{I,N}(\Omega^dk \oplus \Omega^n \Omega^dK)$-filter regular. Using the exact sequence in \ref{F1tor}, we obtain $$r_1^{I,N/xN}(\Omega^dk)=r_1^{I,N}(\Omega^dk)-1,$$
    $$ r_1^{I,N/xN} (\Omega^n\Omega^dk)= r_1^{I,N} (\Omega^n\Omega^dk)-1.$$
    Since $N/xN$ is perfect of dimension $t-1$, hence by induction hypothesis the result follows.

    (ii) is proved similarly.
\end{proof}

\section{The case when \texorpdfstring{$I=\m$}{I=m}} \phantomsection \label{section 5}
We continue with the notation introduced in the previous sections.
In the preceding section, we proved that for every integer $i\ge 1$,
$$
r_i^{I,N}(M) \le \dim N - 1
\quad \text{and} \quad
s_i^{I,N}(M) \le \dim N - 1.
$$
In this section, we shall show that both of these inequalities become equalities when $I$ is the maximal ideal $ \mathfrak{m} $. Using \ref{basechange} we can assume the residue field $k$ is infinite.

\begin{setup}\normalfont \label{setup'}
   Let $(A,\m,k)$  be Gorenstein local ring, which is not regular and $N$ a perfect  module of dimension $t\geq 1$. For any non-free MCM $A$-module $M$ and a positive integer $i$, we set \begin{align*}
   r_i^N(M)=\deg(n\mapsto\lambda(\Tor_{i}^{A}(M,N/\m ^nN))) ,\\
   s^i_N(M)=\deg(n\mapsto \lambda(\Ext^{i}_{A}(M,N/\m ^nN))).
\end{align*}
\end{setup}

\begin{theorem} \label{tor,I=m} (With hypothesis as in \ref{setup'})
 Let $M$ be a non-free MCM $A$-module, then for any integer $i\ge 1$, we have $r_i^N(M)=t-1$.
\end{theorem}
\begin{proof}
   In view of \ref{etpi=1}, it suffices  to prove the statement for $i=1$. For this we apply induction on $t=\dim N$. Let $t=1$. By \ref{abound} we have $\lambda( \Tor_1^A(M,N/\m^nN))=c$ for $n\gg 0$, where $c\ge 0$ is a constant. We need to prove that $c>0$.

   We tensor the exact sequence
$0\longrightarrow \m^nN/\m^{n+1}N\longrightarrow N/\m^{n+1}N\longrightarrow N/\m^nN\longrightarrow 0$
with $M$ to get the exact sequence
$$
 \Tor_1^A(M,N/\m^nN)\longrightarrow \frac{\m^nN}{\m^{n+1}N}\otimes M\longrightarrow \frac{M\otimes N}{\m^{n+1}(M\otimes N)}\longrightarrow \frac{M\otimes N}{\m^{n}(M\otimes N)}\longrightarrow 0.
$$
 Since $\dim N=1$, so $ \m^nN/\m^{n+1}N \cong k^{e(N)}$. Therefore we have $$\frac{\m^nN}{\m^{n+1}N}\otimes M\cong (\frac{M}{\m M})^{e(N)}.$$
Computing lengths, for $n\gg 0$ we have $$c\ge \mu (M)e(N)-\lambda (\frac{\m^n(M\otimes N)}{\m ^{n+1}(M\otimes N)}). $$ Since  $\dim M\otimes N=1$ (see \ref{homtensorCM}), so for $n\gg 0$ $$\lambda (\frac{\m^n(M\otimes N)}{\m ^{n+1}(M\otimes N)})=e(M\otimes N),$$
therefore we obtain
\begin{equation} \label{eq5}
    c \ge \mu (M)e(N)-e(M\otimes N).
\end{equation}
Next we tensor the exact sequence $0 \longrightarrow \Omega M\longrightarrow A^{\mu (M)} \longrightarrow M\longrightarrow 0$
with $N$ to get the exact sequence
$$
    0=\Tor_1^A(M,N) \longrightarrow \Omega M \otimes N \longrightarrow   N^{\mu(M)} \longrightarrow M\otimes N\longrightarrow 0.
$$
By \ref{homtensorCM} all the above nonzero modules have dimension $1$, therefore we get \begin{equation} \label{eq7}
    \mu (M)e(N)-e(M\otimes N)=e(\Omega M\otimes N)
\end{equation}
Now $\dim \Omega M \otimes N=1$ (see \ref{homtensorCM}), therefore $e(\Omega M\otimes N)>0$, and finally combining this with (\ref{eq5}) and (\ref{eq7}), we obtain $$c\ge e(\Omega M\otimes N)> 0 .$$ So for $t=1$ the result holds.

Assume $t\ge 2$. Let $x\in \m $ be such that it is  $N\oplus (M\otimes N)$-superficial and $xt\in A[\m t]_1$ is $L_1^N(M)$-filter regular. Using the exact sequence in \ref{F1tor} we obtain $r_1^{N/xN}(M)=r_1^N(M)-1$. Since $N/xN $ is perfect of dimension $t-1$, so by induction hypothesis the result follows.
\end{proof}

We now proceed to prove an analogous result for Ext.
\begin{theorem} \label{ext,I=m} (With hypothesis as in \ref{setup'})
    Let $M$ be a non-free MCM $A$-module, then for any integer $i\ge 1$ we have $s^i_N(M)=t-1$.
\end{theorem}
\begin{proof} In view of \ref{etpi=1} it suffices  to prove the statement for $i=1$. For this we apply induction on $t=\dim N$. Let $t=1$. We need to show $s^1_N(M)=0$. Suppose $s^1_N(M)<0$, then
\begin{equation} \label{eq8}
    \Ext^1_A(M,N/\m ^n N)=0 \   \operatorname{for} n\gg 0.
\end{equation}
Consider the exact sequence
$$0\longrightarrow \m^nN/\m^{n+1}N\cong k^{e^(N)}\longrightarrow N/\m^{n+1}N\longrightarrow N/\m^nN\longrightarrow 0,$$
 applying $\Hom_A(M,-)$, we obtain for $n\gg 0$

 \begin{equation*}
 \begin{split}
     0\longrightarrow \Hom _A(&M,k)^{e(N)} \longrightarrow \Hom_A(M,N/\m^{n+1}N)\longrightarrow \Hom_A(M,N/\m^n N)
    \\ & \longrightarrow
     \Ext^1_A(M,k)^{e(N)}\longrightarrow  \Ext^1_A(M,N/\m^n N)=0
      \end{split}
 \end{equation*}
Note that $\beta_i(M)=\dim_k\Ext^i_A(M,k)=\lambda(\Ext^i_A(M,k))$. Therefore computing lengths, for $n\gg 0$ we have
\begin{equation} \label{eq10}
    \lambda(\Hom_A(M,N/\m^{n+1}N))-\lambda(\Hom_A(M,N/\m^nN))=e(N)(\mu (M)-\beta_1(M)).
\end{equation}
Since $\dim \Hom_A(M,N)=1=\dim N$(see \ref{homtensorCM}), so by \ref{Theodorescu} we have ($n \mapsto \lambda(\Hom_A(M,N/\m^nN))$) is a polynomial function of degree $1$, say for $n\gg 0$
\begin{equation} \label{eq11}
    \lambda(\Hom_A(M,N/\m^nN))=c_0n+c_1
\end{equation}
with $c_0>0$. We claim  that $c_0=e(\Hom_A(M,N))$. Assume this for now. Then from (\ref{eq10}) we have
\begin{equation} \label{eq12}
   e(\Hom_A(M,N))= e(N)(\mu (M)-\beta_1(M)).
\end{equation}
 Now we apply $\Hom_A(-,N)$ to the exact sequences
 $$
     0\longrightarrow \Omega M\longrightarrow A^{\mu (M)}\longrightarrow M\longrightarrow 0 ,$$
$$ 0\longrightarrow \Omega ^2 M\longrightarrow A^{\beta_1 (M)}\longrightarrow \Omega M\longrightarrow 0,
 $$
and use \ref{torext0} to obtain the following exact sequences
$$0\longrightarrow \Hom_A(M,N)\longrightarrow N^{\mu (M)}\longrightarrow \Hom_A(\Omega M,N)\longrightarrow 0,$$
$$0\longrightarrow \Hom_A(\Omega M,N)\longrightarrow N^{\beta_1 (M)}\longrightarrow \Hom_A(\Omega^2 M,N)\longrightarrow 0.$$
Since all nonzero modules in the above two exact sequences have dimension $1$, therefore we obtain
$$\mu (M) e(N)=e(\Hom_A(M,N))+e(\Hom_A(\Omega M,N))$$
$$\beta_1(M) e(N)=e(\Hom_A(\Omega M,N))+e(\Hom_A(\Omega^2 M,N))$$
Subtracting we get $$e(N)(\mu (M)-\beta_1(M))=e(\Hom_A(M,N))-e(\Hom_A(\Omega^2 M,N)),$$
and from (\ref{eq12}), we have $e(\Hom_A(\Omega^2 M,N))=0$, this is not possible, as $\dim \Hom_A(\Omega^2 M,N)=1$. Hence $s^1_N(M)=0$.

Now we prove our claim that $c_0=e(\Hom_A(M,N))$. Let $x\in \m$ such that $x$ is $\Hom_A(M,N)\oplus N$-superficial. So we have $\m ^{n+1}N:_Nx=\m^nN \operatorname{for} n\gg 0$
and
\begin{equation} \label{eq13}
   e(\Hom_A(M,N))=e(\frac{\Hom_A(M,N)}{x\Hom_A(M,N)})=\lambda(\frac{\Hom_A(M,N)}{x\Hom_A(M,N)}).
\end{equation}
 Since $\dim N=1$, for $n\gg 0$ we have an exact sequence $$0\longrightarrow N/\m^nN \xlongrightarrow{\phi_n} N/\m ^{n+1}N\longrightarrow N/xN\longrightarrow 0$$ where $\phi_n(\zeta+\m^nN)=x\zeta +\m^{n+1}N$ for $\zeta \in N$. By applying the functor $\Hom_A(M,-)$ to the preceding short exact sequence and making use of (\ref{eq8}), we deduce that the following sequence is exact for all $n\gg 0$  $$0\longrightarrow \Hom_A(M,N/\m^nN)\longrightarrow \Hom_A(M,N/\m ^{n+1}N)\longrightarrow \Hom_A(M,N/xN)\longrightarrow 0.$$ Computing lengths and using (\ref{eq11}) we obtain
\begin{equation} \label{eq14}
    c_0=\lambda(\Hom_A(M,N/xN)).
\end{equation}
By assumption, $x$ is a regular element on both $N$ and $\Hom_A(M,N)$, therefore
\begin{equation} \label{eq15}
    \Hom_A(M,\frac{N}{xN})\cong \frac{\Hom_A(M,N)}{x\Hom_A(M,N)}.
\end{equation} Finally, putting together (\ref{eq13}), (\ref{eq14}) and (\ref{eq15}), we conclude the proof of the claim.

So we have proved the assertion for the case $t=1$. Assume $t\ge 2$. Let  $y\in \m$ be such that it is $N$-superficial and $yt\in A[\m t]_1$ is $E^1_N(M)\oplus E^2_N(M)$-filter regular. Using the exact sequence in \ref{F2ext} we obtain $s^1_{N/yN}=s^1_{N}(M)-1$. Since $N/yN$ is perfect of dimension $t-1$, so by induction hypothesis the result follows.
\end{proof}
\begin{remark} \label{counterexample} \normalfont
    The conclusion of the above the theorem fails for arbitrary $\m$-primary ideals. Assume the residue field $k$ is infinite. Let $M$ be a non-free MCM $A$-module. Let $\underline{x}=x_1,\ldots,x_d$ be a regular sequence for $A,M$ and $\Omega M$. Set $J=\underline{x}A$. Then, we have $r_1^{J,A}(M)=-1$ (for detailed discussion see \cite[Remark 20] {Hilbertcoefficient}).
\end{remark}

\section{The Stable category \texorpdfstring{$\CMS(A)$}{CMS(A)}} \phantomsection \label{section 6}
Throughout this section we work with the following setup, unless stated otherwise.
\begin{point}\normalfont \textbf{Setup:} \label{setup8} Let $(A,\m,k)$ be a non-regular Gorenstein ring of dimension $d\ge 1$ and $N$ a perfect module of dimension $t\ge 1$ and $I$ is an ideal of definition of $N$.
\end{point}
\begin{point}\normalfont
   For a MCM $A$-module $M$, we define
$$r^{I,N}_\infty(M):=\operatorname{sup}\{r^{I,N}_1(\Omega^sM)|s\in \Z\}.$$
By \ref{Omegadk=Omegan+dk}, we have $r^{I,N}_\infty(\Omega^dk)=r_1^{I,N}(\Omega^dk).$
\end{point}

We now proceed to explore some properties. To that end, we first state a lemma that will play a crucial role (see \cite[Lemma 4.1.]{ThickCMA}).
\begin{lemma} \label{takahashilemma}
    Let $0\longrightarrow U \longrightarrow V \longrightarrow W \longrightarrow 0$ be an exact sequence of $A$-modules. Then there exists an exact sequence $0\longrightarrow \Omega W\longrightarrow U\oplus F \longrightarrow V\longrightarrow 0$, where $F$ is free.
\end{lemma}
This lemma is easily proved by taking an exact sequence $0\longrightarrow \Omega W \longrightarrow  A^{\mu(W)} \longrightarrow W \longrightarrow 0$ and making a pullback diagram.
\begin{point}\normalfont \label{thickproperties} \textbf{Some properties:} (With hypothesis as in \ref{setup8}) For any MCM $A$-module $M$, we have
\begin{enumerate} [\rm (I)]
    \item $r^{I,N}_\infty(M)=r^{I,N}_\infty(\Omega^iM)$ for any integer $i$.
    \item  If $W|M$, that is, $W$ is a direct summand of $M$, then $r^{I,N}_\infty(W)\le r^{I,N}_\infty(M)$.
    \item If $0\longrightarrow U \longrightarrow V \longrightarrow W \longrightarrow 0$ is an exact sequence of MCM $A$-modules, then we have
\begin{enumerate} [\rm (i)]
    \item $r^{I,N}_\infty(V)\le \operatorname{max}\{r^{I,N}_\infty(U),r^{I,N}_\infty(W)\},$
    \item $r^{I,N}_\infty(U)\le \operatorname{max}\{r^{I,N}_\infty(V),r^{I,N}_\infty(W)\}$,
    \item $r^{I,N}_\infty(W)\le \operatorname{max}\{r^{I,N}_\infty(U),r^{I,N}_\infty(V)\}$

\end{enumerate}

\end{enumerate}
\begin{proof}
    \textbf{(I)} Follows directly from definition.\newline
    \textbf{(II)} There exists an $A$-module $V$ such that $M\cong W\oplus V$. Applying Schanuel's lemma to the following two exact sequences
    $0 \longrightarrow \Omega W \oplus \Omega V \longrightarrow A^{\mu(W)+\mu(V)}\longrightarrow M\longrightarrow 0 $
    and $0\longrightarrow \Omega M \longrightarrow A^{\mu(M)} \longrightarrow M \longrightarrow 0 $,
    we obtain $ \Omega W \oplus \Omega V\oplus A^{\mu(M)} \cong \Omega M \oplus A^{\mu(W)+\mu(V)} $. In particular, there exists a free module $F_1$ such that $\Omega W | (\Omega M\oplus F_1)$. By iterating this process repeatedly, for each integer $s\ge 0$, we obtain a free module $F_s$ such that $\Omega^s W| (\Omega^sM\oplus F_s)$.

    For $s<0$, note that $W^*|M^*$, where $(-)^*=\Hom_A(-,A)$. So there exists a free module $G$ such that $\Omega (W^*)|(\Omega(M^*)\oplus G)$, which implies $(\Omega (W^*))^*|((\Omega(M^*))^*\oplus G^*)$. Recall for a MCM $A$- module $M$, the cosyzygy of $M$ is defined to be as $\Omega^{-1}(M)=(\Omega(M^*))^*$ (see \ref{cosyzygy}). So there exists a free module $F_{-1}=G^*$ such that $\Omega ^{-1}W| ((\Omega ^{-1}M\oplus F_{-1})$. By proceeding inductively for each integer $s$, we obtain a free module $F_s$ such that $\Omega^sW|(\Omega^sM \oplus F_s)$, which implies $r_1^{I,N}(\Omega^sW)\le r_1^{I,N}(\Omega^sM)\le r_\infty^{I,N}(M)$ for each $s\in \Z$. Hence we have $r_\infty^{I,N}(W)\le r_\infty^{I,N}(M)$. \newline
    \textbf{(III)} For each integer $s\ge 0$, we have an exact sequence $0\longrightarrow \Omega^sU \longrightarrow \Omega^sV \oplus F_s \longrightarrow \Omega^sW \longrightarrow 0$, where $F_s$ is free. For $s<0$, we apply $(-)^*$ to obtain the following exact sequence
   $ 0\longrightarrow W^*\longrightarrow V^* \longrightarrow U^* \longrightarrow \Ext^1_A(W,A)=0$.
So we have an exact sequence $0\longrightarrow \Omega(W^*) \longrightarrow \Omega(V^*)\oplus G \longrightarrow \Omega(U^*)\longrightarrow 0$, where $G$ is free. Again we apply $(-)^*$ to obtain the exact sequence $0\longrightarrow (\Omega(U^*))^* \longrightarrow (\Omega(V^*))^*\oplus G^* \longrightarrow (\Omega(W^*))^*\longrightarrow \Ext^1_A(\Omega (U^*),A)= 0$, that is, we have an exact sequence $0\longrightarrow \Omega^{-1}U \longrightarrow\Omega^{-1}V\oplus F_{-1}\longrightarrow\Omega^{-1}W\longrightarrow 0$, where $F_{-1}=G^*$ is free. By proceeding inductively for each integer $s$, we obtain an exact sequence $0\longrightarrow \Omega^sU\longrightarrow \Omega^sV\oplus F_s\longrightarrow \Omega^sW\longrightarrow 0$, where $F_s$ is free. Tensoring with $N/I^nN$, for each integer $s$, we obtain an exact sequence $\Tor_1^A(\Omega^sU,N/I^nN)\longrightarrow \Tor_1^A(\Omega^sV,N/I^nN) \longrightarrow \Tor_1^A(\Omega^sW,N/I^nN)$, which yields $r_1^{I,N}(\Omega^sV)\le \operatorname{max}\{r_1^{I,N}(\Omega^sU),r_1^{I,N}(\Omega^sW)$ for each $s\in \Z$. Thus we have (i).

By \ref{takahashilemma}, there exists an exact sequence $0\longrightarrow \Omega W\longrightarrow U\oplus H \longrightarrow V\longrightarrow 0$, where $H$ is free. Thus (ii) follows from (i) and (iii) follows from (ii).
\end{proof}
\end{point}
We now apply the above properties to a distinguished triangle in $\underline{\operatorname{CM}}(A)$ to obtain similar inequalities.

\begin{lemma} \label{triangleinfinity} (With hypothesis as in \ref{setup8})
    Let $X\longrightarrow Y\longrightarrow Z\longrightarrow \Omega^{-1}X$ be a distinguished triangle in $\underline{\operatorname{CM}}(A)$, then we have
    \begin{enumerate} [\rm (i)]
        \item  $r_\infty^{I,N}(X)\le \operatorname{max}\{r_\infty^{I,N}(Y),r_\infty^{I,N}(Z)\}$,
        \item $r_\infty^{I,N}(Y)\le \operatorname{max}\{r_\infty^{I,N}(X),r_\infty^{I,N}(Z)\}$,
        \item $r_\infty^{I,N}(Z)\le \operatorname{max}\{r_\infty^{I,N}(X),r_\infty^{I,N}(Y)\}$.
    \end{enumerate}
\end{lemma}
\begin{proof}
 By the properties of a distinguished triangle in $\underline{\operatorname{CM}}(A)$, we have an exact sequence $0\longrightarrow Y \longrightarrow Z\oplus F\longrightarrow \Omega^{-1}X\longrightarrow 0$, where $F$ is free. Therefore, we apply \ref{thickproperties} to obtain the desired inequalities.
\end{proof}
\begin{point}\normalfont  \textbf{Construction of \textit{thick}$(E)$:}
    Let $E\in \underline{\operatorname{CM}}(A)$. Recall that, \textit{thick}$(E)$ is the thick subcategory of $\underline{\operatorname{CM}}(A)$ defined as the intersection of all thick subcategories of $\underline{\operatorname{CM}}(A)$ that contains $E$. We now recall a concrete construction of thick$(E)$. To this end, we define the following notations; for two additive subcategories
$\chi$ and $\chi'$ of $\underline{\operatorname{CM}}(A)$, we set $\chi*\chi'=\{D\in \underline{\operatorname{CM}}(A):$there exists a triangle $M\longrightarrow D\longrightarrow M'\longrightarrow
\Omega^{-1}M $, where $ M\in \chi$ and $M'\in \chi'\}$ and $<\chi*\chi'>=\{W\in \underline{\operatorname{CM}}(A):$ there exists $s\in \Z$ and $D\in \chi*\chi'$, such that $W|\Omega^sD\}$.

Now, we consider the  additive subcategory $$\operatorname{add}(E)=\{V\in \underline{\operatorname{CM}}(A):V|E^m \ \operatorname{for \ some} \ m\in \N\}$$ of $\underline{\operatorname{CM}}(A)$. We set $\chi_1=\chi_2=\operatorname{add}(E)$, and for $i\ge 3$,  define recursively $$\chi_i=<\chi_{i-1}*\chi_{i-2}>.$$ Then we have
$$\operatorname{thick}(E)=\bigcup_{i\ge 1} \chi_i.$$
\end{point}
\begin{proposition} (With hypothesis as in \ref{setup8}) \label{thickle} Let $E\in \underline{\operatorname{CM}}(A)$. Then for any $M\in \textit{thick}(E)$, we have $r_\infty^{I,N}(M)\le r_\infty^{I,N}(E)$.
\end{proposition}
\begin{proof}
     From the construction of $\operatorname{thick}(E)$, there exists an integer $i\ge 1$, such that $M\in \chi_i$. We proceed by induction on $i$.  If $i=1$ or $2$, there exists an integer $m\ge 1$ such that $M|E^m$, which implies $r_\infty^{I,N}(M)\le r_\infty^{I,N}(E^m)=r_\infty^{I,N}(E)$. Assume that $i\ge 3$ and that the inequality holds for all integers $j<i$. We have $\chi_i=<\chi_{i-1}*\chi_{i-2}>$.  If $D\in \chi_{i-1}*\chi_{i-2}$, by definition there exists a triangle $V\longrightarrow D\longrightarrow W\longrightarrow \Omega^{-1}V$ in $\underline{\operatorname{CM}}(A)$, where $V\in\chi_{i-1}$ and $W\in \chi_{i-2}$. Thus, by induction hypothesis and \ref{triangleinfinity}, we have $r_\infty^{I,N}(D)\le r_\infty^{I,N}(E)$. Finally, there exists an integer $s$ and $D\in\chi_{i-1}*\chi_{i-2}$ such that $M|\Omega^sD$. Hence, applying \ref{thickproperties} yields the desired inequality.
\end{proof}
We proceed to study additional thick subcategories of $\underline{\operatorname{CM}}(A)$. To this end, for each integer $q\ge -1$, we define  $$K_q=\{M\in \underline{\operatorname{CM}}(A):r_\infty^{I,N}(M)\le q\}.$$
First, we show that it is indeed a thick subcategory of $\underline{\operatorname{CM}}(A)$.
\begin{lemma}
    (With hypothesis as in \ref{setup8}) $K_q$ is a thick subcategory $\underline{\operatorname{CM}}(A)$ for each integer $q\ge -1$.
\end{lemma}
\begin{proof}
    We have
    \begin{enumerate} [\rm (i)]
        \item $0\in K_q$.
        \item If $X\cong Y$ and $X\in K_q$, then clearly $Y\in K_q$.
        \item If $X\in K_q$, then by \ref{thickproperties} we have $\Omega X,\Omega^{-1}X\in K_q$.
        \item Let $X\longrightarrow Y\longrightarrow Z\longrightarrow\Omega^{-1}X$ be a distinguished  triangle in $\CMS(A)$. If two of $X,Y,Z$ belong to $K_q$, then so is the third, follows from \ref{triangleinfinity}.
        \item Let $X,Y\in \underline{\operatorname{CM}}(A)$ with $X\oplus Y \in K_q$. So we have
        $$\operatorname{max}\{r_\infty^{I,N}(X),r_\infty^{I,N}(Y)\}=r_\infty^{I,N}(X\oplus Y) \le q,$$
        which yields $X,Y\in K_q$.
    \end{enumerate}
    Thus, $K_q$ is a thick subcategory of $\underline{\operatorname{CM}}(A)$.
\end{proof}
Next, we prove that the thick subcategory $K_q$ contains a non-free maximal Cohen–Macaulay module of minimal complexity, when $A$ is complete intersection.
\begin{lemma} \label{Kqcx1}
 (With hypothesis as in \ref{setup8}) Let $(A,\m)$ be a local complete intersection. Then, with the above notation, for each integer $q\ge-1$, either $K_q=0$ or $K_q$ contains an MCM $A$-module $E$ such that $\operatorname{cx}_AE=1$.
\end{lemma}
\begin{proof}
    Follows from the above lemma and \ref{thichkcx1}.
\end{proof}

\begin{point}\normalfont
We have a decreasing chain of thick subcategories of $\underline{\operatorname{CM}}(A)$
$$\underline{\operatorname{CM}}(A)=K_{t-1}\supseteq K_{t-2}\supseteq\ldots \ldots\supseteq K_0\supseteq K_{-1}\supseteq 0.$$
\end{point}
Set $$l=\operatorname{min}\{q\ge-1|K_q\neq 0\}.$$
Note that for every non-zero MCM $A$-module $E\in K_l$, we have $r_\infty^{I,N}(E)=l$. Therefore, for each $0\ne E\in K_l$, we obtain \begin{equation} \label{eq28}
   r_\infty^{I,N}(E)=\operatorname{inf\{r_\infty^{I,N}(M)|0\neq M\in \underline{CM}(A)\}}.
\end{equation}
As a corollary of the above results, we obtain the following.
\begin{theorem} \label{cx1CI}

    Let $(A,\m)$ be a complete intersection local ring of dimension $\ge1$ and $N$ a perfect $A$-module of dimension $\ge1$ and $I$ an ideal of definition of $N$. Then, with the above notation, there exists a MCM $A$-module $E$ with of \textit{complexity} 1 such that $$r_1^{I,N}(E)=\inf \{r_\infty^{I,N}(M)|0\ne M \in \CMS(A)\}.$$

\end{theorem}
\begin{proof} Let $l$ be as above. Since $K_l$ is a non-zero thick subcategory of $\CMS(A)$, so $K_l$ contains a MCM $A$-module $W$ of complexity 1 (see \ref{thichkcx1}). From equation (\ref{eq28}), we have $$r_\infty^{I,N}(W)=\operatorname{inf\{r_\infty^{I,N}(M)|0\neq M\in \underline{CM}(A)\}}.$$ Now there exists an integer $s$, such that $r_1^{I,N}(\Omega^sW)=r_\infty^{I,N}(W)$. It is easy to see that $\cx_A(\Omega^sW)=1$ as $\cx_AW=1$. Thus $E=\Omega^sW$ suffices.
\end{proof}
\begin{remark} \normalfont Let $E$ be a MCM $A$-module with no free summand such that $\cx_AE=1$. It follows from \cite[Theorem 9.2.1]{cx1} that $E$ is 2-periodic. Consequently $\Omega^sM\cong \Omega^{s-2}M$ for every integer $s$; use $E\cong \Omega \Omega^{-1}E$. By induction on $t=\dim N$, together with an argument similar to that in \cite[Lemma 2.2 and Theorem 5.1]{HypersurfaceI}, we obtain for each integer $i\ge 1$ that
$r_i^{I,N}(E)=r_1^{I,N}(E)=s^1_{I,N}(E)=s^i_{I,N}(E).$
\end{remark}

\section{Regularity and Depth} \phantomsection \label{section 7}
In this section, we investigate situations in which the upper bound in \ref{abound} is not attained. We also explore further applications within the stable category $\CMS(A)$. Throughout we work with the following setup.
\begin{setup}\normalfont \label{setup6} Let $(A,\m,k)$ be a non-regular Gorenstein ring of dimension $d\ge 1$ and $N$ a perfect module of dimension $t\ge 1$ and $I$ is an ideal of definition of $N$. Recall, for a MCM $A$-module $M$, we set
$$r^{I,N}_\infty(M)=\operatorname{sup}\{r^{I,N}_1(\Omega^sM)|s\in \Z\}.$$

\end{setup}

\begin{point}\normalfont
    Let $M$ be a MCM $A$-module, then $M\otimes N$ is a CM $A$-module of dimension $t$ (see \ref{homtensorCM}). Therefore invoking \ref{b_j<a_j+1}, for all $j=0,\ldots,t-1$  we have
    \begin{equation} \label{eq32'}
        \operatorname{end}( H^j_{\mathfrak{M}}(L^I(M \otimes N))) \le \operatorname{end}(  H^{j+1}_{\mathfrak{M}}(G_I(M \otimes N))).
    \end{equation}
We tensor the surjective map $A^{\mu(M)}\longrightarrow M \longrightarrow 0$ with $N$ to obtain the surjective map $N^{\mu(M)}\longrightarrow M\otimes N\longrightarrow 0$. So for each $n\ge 0$, we have a surjective map
    $$(\frac{I^nN}{I^{n+1}N})^{\mu(M)} \longrightarrow \frac{I^n(M\otimes N)}{I^{n+1}(M\otimes N)} \longrightarrow 0,$$
which induces a surjective map of $\mathcal{R}(I)$-modules
$$G_I(N)^{\mu(M)}\longrightarrow G_I(M\otimes N)\longrightarrow0.$$
Since $\operatorname{dim}N=t=\operatorname{dim} M\otimes N$, therefore we have a surjective map in local cohomology  $$ H^{t}_{\mathfrak{M}}(G_I( N))^{\mu(M)} \longrightarrow H^{t}_{\mathfrak{M}}(G_I(M \otimes N)) \longrightarrow 0,$$
which implies $ \operatorname{end}( H^{t}_{\mathfrak{M}}(G_I(M \otimes N)))\le \operatorname{end} (H^{t}_{\mathfrak{M}}(G_I( N))).$ Thus from equation (\ref{eq32'}) we obtain
$$\operatorname{end}( H^{t-1}_{\mathfrak{M}}(L^I(M \otimes N))) \le \operatorname{end}(  H^{t}_{\mathfrak{M}}(G_I(N))).$$ Note the right-hand side of the above equation depends only on $I,N$. We now generalize this pattern in the following result.
\end{point}
\begin{lemma} \label{boundLI} (With hypothesis as in \ref{setup6}) Let $M$ be a MCM $A$-module such that $r_i^{I,N}(M)\le t-q$ for some integer $q\ge 1$ and for every integer $i\ge 1$. There exits an integer $\eta$, depending only on $I$ and $N$, such that for every integer $j=t-q,\ldots,t-1$ we have
$$\operatorname{end} (H^j_{\mathfrak{M}}(L^I(M \otimes N)))\le \eta.$$
\end{lemma}
\begin{proof}
 We proceed by induction on $q$. For $q=1$, the assertion precisely from the discussion above. Assume $q\ge 2$. Suppose that the result holds for $q-1$, that is, there exists an integer $\eta'$ (depending only on $I$ and $N$) such that for every integer $j=t-q+1,\ldots,t-1$, we have
    \begin{equation}
        \operatorname{end}(H^j_\mathfrak{M}(L^I(W\otimes N))) \le \eta'.
    \end{equation}
 We tensor the exact sequence $0\longrightarrow \Omega M\longrightarrow A^{\mu(M)}\longrightarrow M \longrightarrow 0$ with $N/I^{n+1}N$, for each $n\ge 0$ to obtain the following exact sequence of $\mathcal{R}(I)$-modules $$
   0\longrightarrow L_1^{I,N}(M)\longrightarrow L^I(\Omega M\otimes N) \xlongrightarrow{\phi} L^I(N)^{\mu(M)} \longrightarrow L^I(M\otimes N) \longrightarrow 0. $$
Set $C=\operatorname{Im}\phi$. By hypothesis we have $r_1^{I,N}(M)\le t-q$. Now $L_1^{I,N}(M)=\bigoplus _{n\ge 0}\Tor_1^A(M,N/I^{n+1}N)$ is a finitely generated graded $\mathcal{R}(I)$-module, so $\dim L_1^{I,N}(M)=r_1^{I,N}(M)+1\le t-q+ 1$, thus $H^j_\mathfrak{M}(L_1^{I,N}(M))=0$ for all $j\ge t-q+2$. From
the above exact sequence of $\mathcal{R}(I)$-modules, we have the following exact sequence of graded local cohomology $H^{t-q+1}_\mathfrak{M}(L^I((\Omega M)\otimes N))\longrightarrow H^{t-q+1}_\mathfrak{M}(C)\longrightarrow H^{t-q+2}_\mathfrak{M}(L^{I,N}_1(M))=0$. Note that $r_i^{I,N}(\Omega M)=r_{i+1}^{I,N}(M)\le t-q \le t-q+1$ for every integer $i\ge 1$. Hence by induction hypothesis we have
$$\operatorname{end}(H^{t-q+1}_\mathfrak{M}(C))\le \operatorname{end}(H^{t-q+1}_\mathfrak{M}(L^I((\Omega M)\otimes N))\le \eta'.$$
We also have another exact sequence of graded local cohomologies $$H^{t-q}_\mathfrak{M}(L^I(N))^{\mu (M)}\longrightarrow H^{t-q}_\mathfrak{M}(L^I(M\otimes N))\longrightarrow H^{t-q+1}_\mathfrak{M}(C),$$
which yields $$\operatorname{end}(H^{t-q}_\mathfrak{M}(L^I(M\otimes N)))\le \operatorname{end}(H^{t-q}_\mathfrak{M}(L^I(N)))+\eta'. $$
Thus, $\eta=\operatorname{end}(H^{t-q}_\mathfrak{M}(L^I(N)))+\eta'$ suffices.
\end{proof}
As a corollary of the above result, we obtain the following.

\begin{corollary} (With hypothesis as in \ref{setup6}) Let $E$ be a MCM $A$-module such that $r_\infty^{I,N}(E)\le t-q$ for some integer $q\ge 1$. Then there exits an integer $\eta$, depending only on $I$ and $N$, such that for every MCM $A$-module $M\in thick(E)$ and for every integer $j=t-q,\ldots,t-1$ we have
$$\operatorname{end} (H^j_{\mathfrak{M}}(L^I(M \otimes N)))\le \eta.$$
\end{corollary}
\begin{proof}
    Let $M$ be a MCM $A$-module in thick$(E)$. From \ref{thickle}, we obtain $r_\infty^{I,N}(M)\le r_\infty^{I,N}(E)\le t-q$. In particular, we have $r_i^{I,N}(M)\le t-q$ for $i\ge1$. Thus the result follows.
\end{proof}
Next, we recall a result of Trung and present a generalized version in our setting. For this, the following theorem is essential.
\begin{theorem} \label{boundonGI}
    (with hypothesis as in \ref{setup6}) Let $E$ be a MCM $A$-module such that $r_\infty^{I,N}(E)< t-q$ for some integer $q\ge 1$. Then there exists an integer $\eta$, depending only on $I$ and $N$, such that for every MCM $A$-module $M\in \operatorname{thick}(E)$ and for every integer
 $j=t-q+1,\ldots,t-1$ we have $$a_j(G_I(M\otimes N))\le \eta.$$
\end{theorem}
\begin{proof}
    From the short exact sequence $\mathcal{R}(I)$-modules $$0\longrightarrow G_I(M\otimes N)\longrightarrow L^I(M\otimes N)\longrightarrow L^I(M\otimes N)(-1)\longrightarrow 0,$$
for each integer $j$, we have the following exact in local cohomology $$H^{j-1}_{\mathfrak{M}}(L^I(M\otimes N))(-1)\longrightarrow H^{j}_{\mathfrak{M}}(G_I(M\otimes N)) \longrightarrow H^{j}_{\mathfrak{M}}(L^I(M\otimes N)).$$
Hence from \ref{boundLI}, there exists an integer $\eta'$ (depending only on $I$ and $N$), such that  for each integer  $j=t-q+1,\ldots,t-1$ and $n>\eta'+1$ ,we obtain $H^{j}_{\mathfrak{M}}(G_I(M\otimes N))_n=0$. Thus $\eta=\eta'+1$ serves the purpose.
\end{proof}
We are now in a position to present a generalized version of Trung’s result in our setting. The following result is due to Trung (see \cite[Proposition 3.2]{Trung}).
\begin{proposition} \label{a_d<=red}
    Let $(A,\m,k)$ be a $d$-dimensional Noetherian ring and $I$  an $\m$-primary ideal. Assume the the residue field $k$ is infinite. Then for a finite $A$-module $M$ of dimension $d$, we have $a_d(G_I(M))+d\le \operatorname{red}(I)$, where $\operatorname{red}(I)$ denotes the reduction number of $I$.
\end{proposition}
\begin{proof}
   The case $M=A$ was proved in \cite[Theorem 18.3.12]{Brodman-Sharp} and an identical argument yields the result for arbitrary  $A$-module $M$ of dimension $d$ . The key observation is that any (minimal) reduction of $I$ with respect to $A$ is also a (minimal) reduction of $I$ with respect to $M$.
\end{proof}
 As a corollary of Theorem \ref{boundonGI}, we obtain a generalized version of the above result.
\begin{theorem} \label{Trunggen}
    Let $(A,\m,k)$ be a non-regular Gorenstein local ring of dimension $d\ge 1$ and $I$ an $\m$-primary ideal. Let $E$ be a MCM $A$-module. If $r_\infty^{I,A}(E)\le d-q$ for some integer $q\ge 1$, then there exists an integer $\eta$, depending only on $I$ and $A$, such that for every MCM $A$-module $M\in \operatorname{thick}(E)$, we have
    $$\operatorname{max}\{a_d(G_I(M)),\ldots, a_{d-q+1}(G_I(M))\}\le \eta.$$
\end{theorem}
\begin{proof}
    From \ref{boundonGI} for $N=A$, we obtain an integer $\eta'$ (depending only on $I$ and $A$) such that for each integer $j=d-q+1,\ldots,d-1$, we have $a_j(G_I(M))\le \eta'$. Using \ref{basechange}, we can assume the residue field $k$ is infinite. So invoking \ref{a_d<=red}, we have $a_d(G_I(M)) \le \operatorname{red}(I)$. Hence $\eta=\eta'+\operatorname{red}(I)$ suffices.
\end{proof}
\begin{example}
See \cite[15.6]{CI} for examples of complete intersection and modules $M$ and ideals $I$ with $r_{\infty}^{I, A} = -1$ and $\operatorname{pdim} I^n = \infty$ for all $n \geq 1$.
\end{example}
Next we recall a result from \cite[Proposition 5.2.]{Part1}
\begin{proposition} \label{5.2 Part I}
  Let $(A,\m)$ be a Noetherian local ring, $X$ a finite $A$-module and let $I$ be an ideal of $A$ with $\operatorname{grade}(I,X)>0$. Set $\mathfrak{M}=\m\oplus \mathcal{R}(I)_+$ the *maximal ideal of $\mathcal{R}(I)$. Let $s\le \operatorname{depth}X-1$. Then $$H^i_\mathfrak{M}(L^I(X))=0\  \operatorname{for}\ i=0,\ldots,s\ \operatorname{iff}\ H^i_\mathfrak{M}(G_I(X))=0 \  \operatorname{for}\ i=0,\ldots,s.$$
\end{proposition}

 Let $(A,\m)$ be a Noetherian local ring. For a finitely generated module $X$ and an ideal $I$ with $\lambda(X/IX)<\infty$, we set
$$\depth G_I(X):=\grade(G_I(A)_+,G_I(X))=\grade(\mathfrak{M}_{G_I(A)},G_I(X)),$$
where $\mathfrak{M}_{G_I(A)}=\m \oplus G_I(A)_+$ is the *maximal ideal of $G_I(A)$. Note that $H^i_\mathfrak{M}(G_I(X))=H^i_{G_I(A)_+}(G_I(X))$, since $I$ is an ideal of definition of $X$.
\begin{point} \normalfont (With hypothesis as in \ref{setup6}) Suppose $r_1^{I,N}(\Omega^dK)=-1$, from the proof \ref{-1case} we get $\pdim_A(N/I^nN)<\infty $ for $n\gg 0$.

The next theorem addresses the following situation where $\pdim_A(N/I^{n+1}N)<\infty$ for all $n\ge 0$.

\end{point}

\begin{theorem} \label{depthdueto-1}
   (With hypothesis as in \ref{setup6}) Suppose $\pdim_A(N/I^{n+1}N)<\infty$ for all $n\ge 0$, then for any MCM $A$-module $M$ we have $\depth G_I(M\otimes N)\ge \depth G_I(N).$
\end{theorem}
\begin{proof} It is enough to prove that, if $\depth G_I(N)\ge s$ then $\depth G_I(M\otimes N)\ge s$. We proceed by induction on $s$. For $s=0$, there is nothing to prove. Assume $s\ge 1$ and also assume that our claim holds for $s-1$. Suppose $\depth G_I(N)\ge s$. By \ref{5.2 Part I} we have
    \begin{equation} \label{eq24}
        H^i_\mathfrak{M}(L^I(N))=0  \operatorname{for} \ i=0,\ldots,s-1.
    \end{equation}
    Next we tensor the exact sequence $0\longrightarrow M\longrightarrow A^\tau\longrightarrow \Omega^{-1}M\longrightarrow 0$ with $N/I^{n+1}N$ to obtain the following exact sequence $$\Tor_1^A(\Omega^{-1}M,\frac{N}{I^{n+1}N})\longrightarrow \frac{M\otimes N}{I^{n+1}(M\otimes N)}\longrightarrow \frac{N}{I^{n+1}N}^{\tau}\longrightarrow \frac{(\Omega^{-1}M)\otimes N}{I^{n+1}((\Omega^{-1}M)\otimes N)}\longrightarrow 0.$$
    By hypothesis $\pdim_A(N/I^{n+1}N)<\infty$ for $n\ge 0$, so by \ref{torext0} we have $\Tor_1^A(\Omega^{-1}M,N/I^{n+1}N)=0$ for all $n\ge 0$. Therefore we obtain an exact sequence of $\mathcal{R}(I)$-modules
    \begin{equation} \label{eq25}
      0\longrightarrow L^I(M\otimes N)\longrightarrow L^I(N)^\tau\longrightarrow L^I((\Omega^{-1}M)\otimes N) )\longrightarrow 0.
    \end{equation}
    Now by induction hypothesis $\depth G_I((\Omega^{-1}M)\otimes N))\ge s-1$, so by \ref{5.2 Part I} we have
    \begin{equation} \label{eq26}
     H^i_\mathfrak{M}(L^I((\Omega^{-1}M)\otimes N))=0 \ \operatorname{for} \ i=0,\ldots,s-2.
    \end{equation}
    Using the exact sequence in (\ref{eq25}), we obtain the following exact sequence in local cohomology $$H^i_\mathfrak{M}(L^I((\Omega^{-1}M)\otimes N))\longrightarrow H^{i+1}_\mathfrak{M}(L^I(M\otimes N))\longrightarrow H^{i+1}_\mathfrak{M}(L^I((\Omega^{-1}M)\otimes N)).$$
    Hence by (\ref{eq24}) and (\ref{eq26}) we have $H^i_\mathfrak{M}(L^I(M\otimes N))=0$ for $i=0,\ldots,s-1$. Finally \ref{5.2 Part I} implies that $H^i_\mathfrak{M}(G_I(M))=0$ for $i=0,\ldots,s-1$ and thus $\depth G_I(M\otimes N)\ge s$. This completes the proof.
\end{proof}
Applying the above theorem with $N=A$, we obtain the following corollary.
\begin{corollary}
  Let $(A,\m)$ be a Gorenstein local ring and $I$ an $\m$-primary ideal. Suppose $A/I^{n+1}$ has finite projective dimension for all $n\ge 0$, then for any MCM $A$-module $M$ we have $\depth G_I(M)\ge \depth G_I(A).$
\end{corollary}
\section{Tor-vanishing property and Test module} \phantomsection \label{final section}
 In this section, our aim is to explore cases in which the inequality established in Theorem \ref{abound} becomes an equality.
\begin{setup}\normalfont  \label{setup2}
    Throughout this section, we work with the following setup: \\
 Let $(A,\m,k)$ be a non-regular Gorenstein local ring of dimension $d\ge 1$, with uncountable residue field $k$. Let $N$ be a perfect module of dimension $t\ge 1$ and $I$ an ideal of definition of $N$.
\end{setup}
  In this section our goal is to prove that, for every non-free MCM $A$-module $M$, there exists infinitely many integers $i$ such that
  $r_i^{I,N}(M)=r_1^{I,N}(\Omega^dk)$, under additional assumptions on $A$. By \ref{abound} we have that  $  r_i^{I,N}(M)\le r_1^{I,N}(\Omega^dk)$ for all $i\ge 1$. We first address the case when $r_1^{I,N}(\Omega^dk)=0$. We now recall the definition of a test module.
\begin{definition} \normalfont
    A finite $A$-module $X$ is called a \textit{test module}, if for all $A$-modules $Y$ with $\Tor_{i\gg 0}^A(X,Y)=0$ have finite projective dimension.
\end{definition}
\begin{example} \label{egtestmodule} \normalfont Test modules are abundant.
\begin{enumerate} [\rm (i)]
    \item If $M$ is a test module, then so is $\Omega^nM$ for each integer $n\ge 0$.
    \item  Let $(A,\m,k)$ be a local ring, then the residue field $k$ is a test module. So $\Omega^nk$ is a test module for each integer $n\ge 0$.
    \item Let $(A,\m,k)$ be a local ring with infinite residue field $k$. Let $M$ be a finite $A$-module with $\operatorname{depth}(M)>0$. Then $\widetilde{\m^nM}$ is a test module for all $n\ge1$. As $\widetilde{\m^nM}=\m^nM$ for $n\gg 0$, so $\m^nM$ is a \textit{test module} for all $n\gg 0$ (see \cite[1.5.1, Definition 2.1, Proposition 2.2 and Proposition 2.3(5)]{testmodulefull}.
    \item  Let $(A,\m)$ be a local complete intersection. Then a finite $A$-module $M$ is a test module if and only if  $\operatorname{cx}_AM=\operatorname{codim}A$ (see \cite[Proposition 2.7]{testmodulecx}). Therefore, if $\operatorname{codim}A=1$, then every finite $A$-module $M$ of infinite projective dimension, is a test module. In particular, every non-free MCM $A$-module is a test module, if $A$ is hypersurface.
    \item  Let $(A,\m)$ be a local ring and $I$ an $\m$-primary, integrally closed ideal. Then $A/I$ is a test module (see \cite[Corollary 3.3]{testmoduleic}).
\end{enumerate}
\end{example}
We are now ready to present and prove the main result of this section.
\begin{theorem} \label{infinitei}
    (With hypothesis as in \ref{setup2}) Suppose $M$ is a non-free MCM $A$-module, which is also a \textit{test module}. Then there exists infinitely many integers $i\ge 1$ such that $r_i^{I,N}(M)=r_1^{I,N}(\Omega^dk).$
\end{theorem}
\begin{proof}
 If $r_1^{I,N}(\Omega^dk)=-1$, the assertion follows from \ref{abound}. so we assume $r_1^{I,N}(\Omega^dk)\ge 0$. We proceed  by contradiction. Suppose there exist an integer $i_0$ such that $r_i^{I,N}(M)<r_1^{I,N}(\Omega^dk)$ for all $i>i_0$. Now $\Omega^jM$ is a test module for each $j\ge1$. Therefore, replacing  $M$ by $\Omega^{i_0}M$, we can assume that $r_i^{I,N}(M)<r_1^{I,N}(\Omega^dk)$ for all $i\ge1$.

  Let $r_1^{I,N}(\Omega^dk)=0$. So we have $r_i^{I,N}(M)=-1$ for all $i\ge 1$. Hence for each $i\ge 1$, there exists an integer $n(i)$ such that
  \begin{equation} \label{eq34}
      \Tor_i^A(M,N/I^{n+1}N)=0 \ \operatorname{for} \ \operatorname{all} \ n\ge n(i).
  \end{equation}
For each $i\ge0$, tensoring the exact sequence $0\longrightarrow\Omega^{i+1}(M)\longrightarrow A^{\beta_i(M)}\longrightarrow \Omega^iM\longrightarrow 0$ with  $N/I^{n+1}N$, we obtain the following injective map graded $\mathcal{R}(I)$-modules
$$0\longrightarrow L^{I,N}_1(\Omega^iM)\longrightarrow L^I(\Omega^{i+1}M\otimes N).$$
Since the residue field $k$ is uncountable, we can choose an element $x\in I$ such that $x\in I$ is $\Omega^{i+1}M\otimes N$-superficial with respect to $I$ and $xt\in A[It]_1$ is $L^{I,N}_1(\Omega^iM)$-filter regular for all $i\ge 0$. Note that $L^{I,N}_1(\Omega^iM)=L^{I,N}_{i+1}(M)$. Thus for each $n\ge 0$  and $i\ge 1$ we have the following commutative diagram

$$
\begin{tikzcd}
0 \arrow[r] & L_i^{I,N}(M)_n \arrow[r] \arrow[d, "xt"] & L^I(\Omega^{i+1}M\otimes N)_{n} \arrow[d, "xt"] \\
0 \arrow[r] & L_i^{I,N}(M)_{n+1}  \arrow[r] & L^I(\Omega^{i+1}M\otimes N)_{n+1}
\end{tikzcd}
$$
Invoking \ref{superficialstarting}, we have the following injective map for each $i\ge 1$ and $n\ge \eta$ (obtained in Lemma \ref{boundLI})
$$0\longrightarrow L^I(\Omega^{i+1}M\otimes N)_{n}=\frac{\Omega^{i+1}M\otimes N}{I^{n+1}(\Omega^{i+1}M\otimes N)} \xlongrightarrow{xt}  L^I(\Omega^{i+1}M\otimes N)_{n+1}=\frac{\Omega^{i+1}M\otimes N}{I^{n+2}(\Omega^{i+1}M\otimes N)},$$
thus for each $i\ge 1$ and $n\ge \eta$, we have an injective map
$$0\longrightarrow \Tor_i^A(M,N/I^{n+1}N)\longrightarrow \Tor_i^A(M,N/I^{n+2}N).$$
Therefore from equation (\ref{eq34}), for each $i\ge 1$ and $n\ge \eta$, we obtain $\Tor_i^A(M,N/I^{n+1}N)=0$. Since $M$ is a test module, so we have $\pdim_A(N/I^{n+1}N)<\infty$ for all $n\ge \eta$. By \ref{torext0}, we obtain $\Tor_1^A(\Omega^dk,N/I^{n+1}N)=0$ for all $n\ge \eta$, which implies $r_1^{I,N}(\Omega^dk)=-1$, a contradiction.

Let $r_1^{I,N}(\Omega^dk)=r\ge 1$. Since the residue field $k$ is uncountable, we can choose a sequence of elements $\underline{x}=x_1,\ldots,x_r\in I$ such that $\underline{x}\in I$ is $N,\Omega^dk\otimes N, \Omega^iM\otimes N$-superficial sequence and $\underline{x}t\in A[It]_1$ is $L^{I,N}_1(\Omega^dk),L^{I,N}_1(\Omega^iM)$-filter regular sequence for each $i\ge 0$. Note that $\bar{N}=N/\underline{x}N$ is a perfect $A$-module and $\dim \bar{N}\ge 1$. Using the assumptions on the sequence $\underline{x}$ and imitating the argument in \ref{-1argu} and  applying the exact sequence in \ref{F1tor} repeatedly, we obtain, for each $i\ge 1$, that
$$r_1^{I,\bar{N}}(\Omega^dk)=0 \ \operatorname{and} \ r_i^{I,\bar{N}}(M)=r_i^{I,N}(M)-r.$$ Since $r_1^{I,\bar{N}}(\Omega^dk)=0$, hence by the previous case there exists infinitely many integers $i\ge1$, such that $r_i^{I,\bar{N}}(M)=0$, thus for these integers $i\ge 1$, we have $r_i^{I,N}(M)=r=r_1^{I,N}(\Omega^dk)$. This completes the proof.
\end{proof}
We now give the definition of \textit{Tor-vanishing property} of a ring.
\begin{definition} \normalfont
A Noetherian ring $A$ is said to satisfy the \textit{Tor-vanishing property} if for finite $A$-modules $M,N$ with $\Tor_{i\gg 0}^A(M,N)=0$ implies that either $M$ or $N$ has finite projective dimension.
\end{definition}
\begin{example} \label{egtvp} \normalfont  Let $(A,\m)$ be Gorenstein local ring of minimal multiplicity, that is, $e(A)=\operatorname{emdim}A-\dim A+2$ and assume $\operatorname{codim}A\ge 3$ (so $A$ is not a complete intersection). Then $A$ satisfies Tor-vanishing property (see \cite[Theorem 3.6.]{Tor-vanishing GO})

For more examples, we refer to \cite{Tor-vanishing} (see Theorem 4.8).
\end{example}

\begin{remark} \normalfont
 If the ring $A$ satisfies the Tor-vanishing property, then any finite
 $A$-module with infinite projective dimension, is a test module. In particular every non-free MCM $A$-module is a test module. Hence, as a corollary of Theorem \ref{infinitei}, we obtain the following result.
\end{remark}

\begin{corollary} Let $(A,\m,k)$ be a non-regular Gorenstein local ring of dimension $d\ge 1$, with uncountable residue field $k$, satisfying the \textit{Tor-vanishing property}. Let $N$ be a perfect module of dimension $t\ge 1$ and $I$ an ideal of definition of $N$. Then for every non-free MCM $A$-module $M$, there exists infinitely many integers $i\ge 1$ such that $$r_i^{I,N}(M)=r_1^{I,N}(\Omega^dk).$$
     \end{corollary}


\providecommand{\bysame}{\leavevmode\hbox to3em{\hrulefill}\thinspace}
\providecommand{\MR}{\relax\ifhmode\unskip\space\fi MR }
\providecommand{\MRhref}[2]{
  \href{http://www.ams.org/mathscinet-getitem?mr=#1}{#2}
}

\end{document}